\documentclass[10pt]{article}
\makeatletter
\DeclareOldFontCommand{\rm}{\normalfont\rmfamily}{\mathrm}
\DeclareOldFontCommand{\sf}{\normalfont\sffamily}{\mathsf}
\DeclareOldFontCommand{\tt}{\normalfont\ttfamily}{\mathtt}
\DeclareOldFontCommand{\bf}{\normalfont\bfseries}{\mathbf}
\DeclareOldFontCommand{\it}{\normalfont\itshape}{\mathit}
\DeclareOldFontCommand{\sl}{\normalfont\slshape}{\@nomath\sl}
\DeclareOldFontCommand{\sc}{\normalfont\scshape}{\@nomath\sc}
\makeatother
\usepackage{typearea}
\typearea{10}
\usepackage{amsmath,amsfonts,amsthm,amssymb}
\usepackage{color}
\usepackage{algorithm}
\usepackage{algcompatible}
\usepackage{framed}
\usepackage{arydshln}
\usepackage{comment}
\usepackage{url}
\usepackage{subcaption}
\newtheorem{theorem}{Theorem}[section]
\newtheorem{proposition}{Proposition}[section]
\newtheorem{lemma}{Lemma}[section]
\newtheorem{corollary}{Corollary}[section]

\theoremstyle{definition}
\newtheorem{definition}[theorem]{Definition}
\newtheorem{remark}[theorem]{Remark}

\numberwithin{equation}{section}

\newcommand{\R}{\mathbb{R}}
\newcommand{\prox}[1]{\mbox{\rm{prox}}_{#1}}
\newcommand{\phic}{\frac{\phi}{c}}

\newcommand{\AL}[3]{L_{{#3}}({#1},{#2})}

\title{The Proximal Method of Multipliers for a Class of Nonsmooth Convex Optimization
}
\author{
T.~Takeuchi\thanks{Graduate School of Mathematical Sciences, The University of Tokyo, Tokyo, Japan. (take@ms.u-tokyo.ac.jp)}
}
\date{}

\begin{document}
\maketitle

\begin{abstract}
This paper develops the proximal method of multipliers for a class of nonsmooth convex optimization.
The method generates a sequence of minimization problems (subproblems).
We show that the sequence of approximations to the solutions of the subproblems converges to a saddle point of the Lagrangian even if the original optimization problem may possess multiple solutions.

The augmented Lagrangian due to Fortin appears in the subproblem. The
remarkable property of the augmented Lagrangian over the standard
Lagrangian is that it is always differentiable, and it is often semismoothly differentiable. This fact allows us to employ a nonsmooth Newton method for computing an approximation to the subproblem. The proximal term serves as the regularization of the objective function and guarantees the solvability of the Newton system without assuming strong convexity on the objective function.
We exploit the theory of the nonsmooth Newton method to provide a rigorous proof for the global convergence of the proposed algorithm.
\end{abstract}
{\footnotesize}

\section{Introduction}
This article develops the proximal method of multipliers for solving
a class of nonsmooth convex optimization
 \begin{equation}\label{min}
   \min_{x\in \R^n} f(x) +\phi(Ex). \tag{$\mathcal{P}$}
 \end{equation}
Throughout this article, we assume that
$f \colon \R^n \rightarrow \R$ is twice continuously differentiable and (not necessarily strongly) convex,
$\phi \colon \R^m \rightarrow \R\cup \{+\infty\}$ is a proper closed convex function, and $E \colon \R^n \rightarrow \R^m$ is a linear operator.
Accurate numerical solvers for nonsmooth convex optimization have
received considerable interest. Among the existing methods, the
alternating direction method of multipliers (ADMM)
\cite{Boyd+ParikhETAL-DistOptiStatLear:11,Eckstein+Bertsekas-Dougsplimethprox:92a},
the augmented Lagrangian method (ALM)
\cite{Ito+Kunisch-AugmLagrmethnons:00, Fortin+Glowinski-AugmLagrmeth:83} and its variant
\cite{Tomioka+Sugiyama-DualLagrMethEffi:09}, and the primal dual
active set (PDAS) method
\cite{Hintermueller+ItoETAL-primactistrasemi:02,Ito+Kunisch-actistrabaseaugm:99}
are predominant approaches. Most of these methods, however, are
designed for objective functions possessing strong convexity, which is
rather restrictive in practical applications. Rockafellar proposed in
\cite{Rockafellar-MonooperaugmLagr:78} the proximal method of
multipliers, which is a variant of the augmented Lagrangian method
that works without the assumption of strong convexity on the objective function. The method is designed for the inequality-constrained problem
\begin{equation}\label{min_ineq}
  \min_{x\in C} f(x) \quad \mbox{satisfying}\quad g_1(x)\le 0,\ldots,g_m(x)\le 0.
\end{equation}
Here, $C$ is a nonempty closed convex subset of $\R^n$, and $f\colon C\longrightarrow \R$ and $g_i \colon C\longrightarrow \R$ for $i=1,\ldots,m$ are lower semicontinuous convex functions. The proximal method of multipliers is an iteration scheme that generates, for any initial pair $(x_0,\lambda_0)$ and the sequence $\{c_k\}$ of nondecreasing positive real numbers, a sequence $(x_k,\lambda_k)$ by
\begin{align}\label{ppm}
\left\{\begin{array}{rl}
\displaystyle  x_{k+1} & \displaystyle \approx \arg\min_{x\in C}\left( \AL{x}{\lambda_k}{c_k} + 1/(2c_k)\Vert x - x_k\Vert^2\right),\\[5pt]
\displaystyle  \lambda_{k+1} & \displaystyle  =\lambda_k + c_k\nabla_\lambda \AL{x_{k+1}}{\lambda_k}{c_k}\\[5pt]
&= \displaystyle \lambda_k + c_k(g(x_{k+1})-\min(0,g(x_{k+1})+\lambda_k/c_k))\\[5pt]
&= \displaystyle\max(0, \lambda_k + c_k g(x_{k+1})).
\end{array}\right.
\end{align}
Here, $g(x)=(g_1(x),\ldots,g_m(x))^\top \in \R^m$ and $\AL{x}{\lambda}{c}$ is the augmented Lagrangian for the inequality-constrained problem defined by
\begin{equation}\label{ag_ineq}
  \AL{x}{\lambda}{c} =f(x)+ \frac{1}{2c}\left(\Vert \max( \lambda+cg(x),0)\Vert^2-\Vert\lambda\Vert^2\right).
\end{equation}
The precise meaning of the approximation ``$\approx$'' is specified depending on the criterion, and
the following one was treated in \cite{Rockafellar-MonooperaugmLagr:78}:
\begin{equation*}
\mbox{dist}(0,\partial \phi_k(x_{k+1})) \le \epsilon_k/c_k\min(1, \Vert (   x_{k+1}-x_k, \lambda_{k+1}-\lambda_k)\Vert^r), \qquad \sum_{k=0}^\infty \epsilon_k < \infty.%\label{stop2}
\end{equation*}
Here, $r$ is either 0 or 1, and $\phi_k(x) = \AL{x}{\lambda_k}{c_k} + 1/(2c_k)\Vert x - x_k\Vert^2$.
Based on the theory of the proximal point algorithm for maximal monotone operators developed in \cite{Rockafellar-Monooperproxpoin:76a}, Rockafellar derived the proximal method of multipliers \eqref{ppm} and established the convergence of the algorithm \cite[Thm.~7]{Rockafellar-MonooperaugmLagr:78}.
The advantage over the method of multipliers (see the Appendix for the
method of multipliers) is that, as mentioned briefly above, even when
the strong convexity of the objective function $f$ is absent and the
primal and dual problems may posses multiple solutions, the
sequences $\{x_k\}$ and $\{\lambda_k\}$ both converge to an optimal
solution of the primal problem and of the dual problem, respectively.

\noindent
\textbf{Application to nonsmooth convex optimization}

It is reasonable to expect that this framework carries over to the problem \eqref{min}.
In fact, Rockafellar noted in the concluding remarks of the article
\cite{Rockafellar-MonooperaugmLagr:78} that the algorithm can be
transferred to more general problems beyond the
inequality-constrained problem and provided brief guidelines on how one may accomplish this task.
In this paper, we follow Rockafellar's remark to obtain the proximal method of multipliers (see Algorithm~\ref{alg:1}) for the problem \eqref{min}, where the minimization problem of the augmented Lagrangian $\AL{x}{\lambda}{c}$ defined by
\begin{equation}\label{ag}
 \AL{x}{\lambda}{c} =f(x) + \phi_c(E x + \lambda/c) -\tfrac{1}{2c}\|\lambda\|^2
\end{equation}
arises naturally. Here, $\phi_c$ is the Moreau envelope defined in section~\ref{sec:moreau}.
The remarkable property of the augmented Lagrangian \eqref{ag} is that it is differentiable and the gradient is Lipschitz continuous, and hence, one can define the (generalized) Jacobian. These facts allow us to design a nonsmooth Newton method for solving the subproblem for $x_{k+1}$ in \eqref{ppm_ag2}.

\noindent
\textbf{Contributions}

This paper studies the proximal method of multipliers for the problem \eqref{min} and Newton's method for solving the sequence of subproblems. Our contributions are summarized as follows:
\begin{enumerate}
  \item We perform Rockafellar's minmax application of the proximal point algorithm to a maximal monotone operator associated with the ordinary Lagrangian of the problem \eqref{min} to derive the proximal method of multipliers. This work aims to complete the aforementioned remark noted by Rockafellar in \cite{Rockafellar-MonooperaugmLagr:78}.
  \item The proximal method of multipliers requires an optimization method for computing an approximation $x_{k+1}$ of the inner minimization problem in \eqref{ppm_ag2}. We employ a nonsmooth Newton method \cite{Pang+Qi-Nonsequa:93} with a line-search procedure to compute the approximation (see Algorithm~\ref{alg:2}).
      The proximal term $1/(2c_k)\Vert x - x_k\Vert^2$ plays several roles in the inner minimization: the quadratic term makes the objective function strongly convex. Consequently, the unique solvability of the problem is ensured. The term also positively affects the local convergence of the Newton method, i.e., the solution of the inner problem will be located in the vicinity of the previous iterate $x_k$, and hence, by starting the inner iteration with $x_k$ as the initial point, the unit step length in the line-search procedure tends to be chosen after a few iterations, and thus, the Newton iteration exhibits fast local convergence.
      This phenomenon was observed in our numerical experiments. We present theoretical evidence for the observation, that is, we show that the sequence generated by Algorithm~\ref{alg:2} globally converges to the unique minimizer of the subproblem at least quotient superlinearly.
\end{enumerate}

The augmented Lagrangian \eqref{ag} first appeared in the article of Fortin~\cite{Fortin-Minisomenon-func:75}, where he proposed a method of multipliers, also referred to as the augmented Lagrangian method:
\begin{align}\label{mm_ag}
\left\{
\begin{array}{l}
\displaystyle x_{k+1} = \arg\min_{x} \AL{x}{\lambda_k}{c},   \\[5pt]
\displaystyle \lambda_{k+1}  = \lambda_k + c(Ex_{k+1}-\mbox{prox}_{\phi/c}(Ex_{k+1}+\lambda_k/c)).
\end{array}\right.
\end{align}
(cf. section~\ref{sec:moreau} for the definition of $\prox{\phic}$.)  However, the article \cite{Fortin-Minisomenon-func:75} has attracted little attention since its appearance.
After more than two decades since his work, the same approach was proposed (rediscovered) by Ito and Kunisch~\cite{Ito+Kunisch-AugmLagrmethnons:00}.
An iterative algorithm that is slightly different in appearance but essentially identical to the augmented Lagrangian method \eqref{mm_ag} was also proposed by Tomioka and Sugiyama in the context of sparsity-regularized minimization problems \cite{Tomioka+Sugiyama-DualLagrMethEffi:09,Tomioka+SuzukiETAL-SupeConvDualAugm:11}.

There are a number of references on the augmented Lagrangian method for convex programming with equality-inequality constraints (e.g., \cite{MR690767,Birgin+Martinez-PracaugmLagrmeth:14}) and on ADMM for \eqref{min}
  (e.g., \cite{Boyd+ParikhETAL-DistOptiStatLear:11,Eckstein+Bertsekas-Dougsplimethprox:92a}). In contrast, to the
best of our knowledge, there are few articles available regarding the augmented Lagrangian method (ALM) developed by Fortin \cite{Fortin-Minisomenon-func:75}  for the problem \eqref{min}.
  To fill this gap, we provide a brief overview of the augmented
  Lagrangian method. Our focus is twofold: the first is to excavate
  the work by Fortin, which has been buried and has been scarcely referred
  to in the subsequence literature since its appearance, and the second is to clarify the difference between ADMM and ALM.
  %
 %In order to fill the gap, we give a brief overview of the method to clarify the difference between ADMM and ALM.
 We also present some previous works closely related to the method,
 namely, the dual augmented Lagrangian (DAL) method
 \cite{Tomioka+Sugiyama-DualLagrMethEffi:09,Tomioka+SuzukiETAL-SupeConvDualAugm:11} and the forward-backward
 Newton method \cite{parikh2013proximal,Patrinos+StellaETAL-ForwtrunNewtmeth:14}. These
 are placed in the Appendix.

The remainder of this paper is organized as follows.
In section 2, we present a brief review of proximal mappings and the augmented Lagrangian, which are the basic building blocks for designing the proximal method of multipliers.
In section 3, we derive the proximal method of multipliers for the problem \eqref{min} from the general proximal point algorithm.
In section 4, we describe a nonsmooth Newton method for solving the inner problem \eqref{ppm_ag2}, and
we show that the algorithm is globally convergent at least quotient superlinearly.
We also illustrate the implementation of the algorithm.
\section{Preliminaries}
For a twice continuously differentiable function $f$, its gradient is
denoted by $\nabla f(x)$, and its Hessian is denoted by $\nabla^2 f(x)$.
A convex function $\phi$ is said to be proper if at least one $x_0$
exists such that $\phi(x_0)<+\infty$ and $\phi(x)>-\infty$ for all $x$. The set of proper, lower semicontinuous, convex functions defined on a space ${\R^n}$ is denoted by $\Gamma_0({\R^n})$. The effective domain
of a function $\phi\in\Gamma_0({\R^n})$ is denoted by $\mbox{dom}(\phi)=\{z\in {\R^n}\mid \phi(z)\mbox{ is finite}\}$, and
it is always assumed to be nonempty.
For a function $\phi\in \Gamma_0(\R^n)$, the convex conjugate $\phi^*$ is defined by $\phi^{\ast}(z^\ast) = \sup_{z\in \R^n}\left( (z^{\ast}, z) - \phi(z) \right)$.
A \textit{subgradient} of $\phi$ at $x\in \R^n$ is $g\in\R^n$ satisfying
\begin{equation*}
\phi(y)\ge \phi(x) + (g, y-x),\quad \forall y \in \R^n.
\end{equation*}
The \textit{subdifferential} of $\phi$ at $x$ is the set of all
subgradients of $\phi$ at $x$, and it is denoted by $\partial \phi(x)$.

Let $\{x_k\}\subset \R^n$ be a sequence converging to a limit $x_{\infty}$. We say that the rate of convergence is quotient superlinear (Q-superlinear) if
$\lim_{k\to \infty}\frac{\Vert x_k- x_{\infty}\Vert}{\Vert x_{k+1}-x_{\infty}\Vert}=0$,
and quotient quadratic (Q-quadratic) if
$\lim_{k\to\infty}\frac{\Vert x_k-x_{\infty}\Vert}{\Vert x_{k+1}-x_{\infty}\Vert^2}<\infty$.
\subsection{Dual problem and the (standard) Lagrangian}\label{sec:dual_Lagrangian}
Let $(\mathcal{D})$ denote the Fenchel-Rockafellar dual problem of \eqref{min} defined by
\begin{equation}\label{max}
\max_{\lambda \in \R^m} -f^\ast(-E^\top \lambda) - \phi^\ast(\lambda). \tag{$\mathcal{D}$}
\end{equation}
Throughout this article, we assume that the following conditions are satisfied:
\begin{itemize}
\item[\textrm{(A1)}] There exists an $x\in \R^n$ such that $x\in \mbox{ri}(\mbox{dom} f)$ and $Ex\in \mbox{ri}(\mbox{dom} \phi)$.
\item[\textrm{(A2)}] There exists a $\lambda\in \R^m$ such that $-\lambda\in \mbox{ri}(\mbox{dom} \phi^\ast)$ and $-E^\top \lambda\in \mbox{ri}(\mbox{dom} f^\ast)$.
\end{itemize}
Here, $\mbox{ri}(S)$ for a set $S$ denotes the relative interior of
$S$, i.e., $\mbox{ri}(S)=\{x\in S \mid \exists \epsilon>0 \mbox{ such
  that }\mbox{aff}(S)\cap B(x,\epsilon) \subseteq S\}$, where
$\mbox{aff}(S)$ is the smallest affine set that contains the set $S$.
These conditions guarantee the existence of optimal solutions both of \eqref{min} and \eqref{max} (see, e.g., \cite[Prop.~2.2.12]{Kanno-Nonsmechconvopti:11}).

Now let $L$ be the ordinary Lagrangian for \eqref{min}
\begin{equation*}%\label{lag}
L(x,\lambda) = \inf_{u} \left(f(x) + \phi(Ex + u) - ( \lambda,u) \right)= f(x) - \phi^\ast(\lambda) +( \lambda,Ex).
\end{equation*}
A pair $(\bar{x},\bar{\lambda})$ is said to be a saddle point of the Lagrangian $L$ if it satisfies
\begin{equation*}
L(\bar{x}, {\lambda}) \le L(\bar{x},\bar{\lambda}) \le L(x,\bar{\lambda}) \quad \forall x \in \R^n,\forall\lambda\in\R^m.
\end{equation*}
The next result links a saddle point of the Lagrangian $L$ with a pair of optimal solutions for \eqref{min} and \eqref{max} (see, e.g.,
\cite[Prop.~2.2.17]{Kanno-Nonsmechconvopti:11},\cite[Thm.~4.35]{Ito+Kunisch-Lagrmultapprvari:08}).
\begin{proposition}\label{prop:saddle}
Let $\Phi(x,\lambda)=f(x)+\phi(Ex-\lambda)$. Assume that $\Phi \colon \R^n\times \R^m\rightarrow \R\cup \{+\infty\}$ is a closed proper convex function.
Then, the following are equivalent:
\begin{itemize}
\item[\textrm{(1)}] $(\bar{x},\bar{\lambda})$ is a saddle point of $L$.
\item[\textrm{(2)}] $\bar{x}$ solves \eqref{min}, $\bar{\lambda}$ solves \eqref{max}, and $\min \eqref{min} = \max \eqref{max}$ holds.
\end{itemize}
\end{proposition}
Since $f$ and $\phi$ are closed proper convex functions, $\Phi$ is
also a closed proper convex function. Proposition~\ref{prop:saddle} together with (A1) and (A2) implies the existence of a saddle point of $L$.
The existence of the saddle point of $L$ is crucial for the proximal method of multipliers.
As we will show in section 3, if the Lagrangian $L$ possesses at least one saddle point,
then the sequence generated by the proximal method of multipliers converges to a saddle point of $L$, which again by Proposition \ref{prop:saddle} implies that the sequence converges to a pair of optimal solutions of \eqref{min} and \eqref{max}.
\subsection{Moreau envelope and proximal operator}\label{sec:moreau}
We recall the definitions and some basic properties of the Moreau envelope and proximal operator.
Let $\phi\in \Gamma_0({\R^m})$ and $c>0$.
The \textit{proximal operator} $ \prox{\phi}: {\R^m}\rightarrow {\R^m}$  is defined as
\begin{align*}
 \prox{\phic}(z)=\arg\min_{u \in {\R^m}} \left(\phi(u) + \tfrac{c}{2}\Vert u-z\Vert^2 \right).
\end{align*}
Since the function to be minimized is strongly convex, it admits a
unique solution, and thus, the proximal operator is well defined.
The value function is called the \textit{Moreau envelope} or \textit{Moreau-Yosida approximation} and is denoted by $\phi_c$:
\begin{equation*}
\phi_c(z) =\phi(\prox{\phic}(z))+\tfrac{c}{2}\Vert \prox{\phic}(z)- z\Vert^2.
\end{equation*}
We list some properties of the Moreau envelope and the proximal operator \cite{Bauschke+Combettes-Convanalmonooper:11}:\\
\textbf{Upper bound.} $0 \le \phi(z)-\phi_c(z) $ for all $z \in {\R^m}$ and all $c>0$.\\
\textbf{Approximation.} $\lim_{c\to \infty}\phi_c(z) = \phi(z)$ for all $z\in {\R^m}$.\\
\textbf{Composition.} If $\phi(z) = a\psi(\alpha z + \beta)+b$ for $\psi\in\Gamma_0(\R^m)$ with $a\in\R,b\in \R^m,\beta\in \R^m$ and $\alpha\neq 0$, then
\begin{equation}\label{composition}
\prox{\phic}(z) = \frac{1}{\alpha}\left(\prox{\frac{a\alpha^2 \psi}{c}}(\alpha z + \beta) - \beta\right)
\end{equation}
\textbf{Affine addition}  If $\phi(z) = \psi(z)+(a,x)+b$ with $a,b\in \R^m$, then
\begin{equation*}%\label{addition}
\prox{\phic}(z) = \prox{\frac{\psi}{c}}(z- a/c).
\end{equation*}
\textbf{Nonexpansiveness.} The proximal operator $\prox{\phic}$ is nonexpansive, that is,
\begin{equation*}
 \Vert \prox{\phic}(z) - \prox{\phic}(w)\Vert^2 \le (\prox{\phic}(z) - \prox{\phic}(w), z-w), \quad \forall z, \forall w\in {\R^m}.
 \end{equation*}
\textbf{Differentiability.} The Moreau envelope $\phi_c$ is Fr\'echet differentiable, and the gradient is given by
\begin{equation*}%\label{phip}
   \nabla\phi_c(z) =c(z - \prox{\phic}(z) ), \quad \forall c>0, \forall z\in {\R^m}.
\end{equation*}
\textbf{Lipschitz continuity of the gradient.} The gradient $\nabla\phi_c $ is Lipschitz continuous with a Lipschitz constant $c$, i.e.,
\begin{equation*}%\label{lip}
\Vert \nabla \phi_c(z) -  \nabla \phi_c(w)\Vert \le c\Vert z -w\Vert,  \quad \forall z ,\forall w \in {\R^m}.
\end{equation*}
\textbf{Decomposition.} The Moreau envelope and the proximal mapping of the conjugate of $\phi$ are related with $\phi_c$ and $\prox{\phic}$, respectively, as
\begin{align*}%\label{decomposition_phi}
\phi_c(z) + (\phi^\ast)_{\frac{1}{c}}\left(cz\right) =\tfrac{c}{2}\Vert z \Vert^2, \qquad %\label{moreau_id} \\
\prox{\phic}(z) + \tfrac{1}{c}\prox{c\phi^{\ast}}\left(cz\right) = z. \nonumber
\end{align*}
\textbf{Subdifferential.} The following conditions are equivalent:
\begin{equation*}
\lambda \in \partial \phi(z) \Longleftrightarrow z - \prox{\phic}(z+\lambda/c)=0 \Longleftrightarrow \phi(z) = \phi_c(z+\lambda/c)-\frac{1}{2c}\Vert \lambda \Vert^2.
\end{equation*}
We refer interested readers to Tables~10.1 and 10.2  in \cite{Combettes+Wajs-Signrecoproxforw:05} for closed-form expressions of a number of frequently used proximal mappings.\\

\subsection{The augmented Lagrangian}
Let us recall the definition of the augmented Lagrangian \eqref{ag}.
The augmented Lagrangian $L_c(x,\lambda)$ is finite for all $x\in {\R^n}$ and $\lambda\in {\R^m}$ even when $\phi$ takes infinity at some point, such as the indicator function on a convex set. Moreover, it is Frechet differentiable with respect to both variables:
\begin{proposition}[\cite{Jin+Takeuchi-Lagroptisystclas:16}]\label{prop:KKT}
Let $c>0$ and $f$ be convex and continuously differentiable, and let $\phi\in\Gamma_0({\R^m})$. The augmented Lagrangian
$L_c$ satisfies the following properties:
\begin{itemize}
 \item[\textrm{(1)}] $L_c$ is finite for all $x\in {\R^n}$ and for all $\lambda \in {\R^m}$.
 \item[\textrm{(2)}] $L_c$ is convex and continuously differentiable
   with respect to $x$, and it is concave and continuously
 differentiable with respect to $\lambda$. Furthermore, for all
 $(x,\lambda)\in {\R^n}\times {\R^m}$ and for all $c>0$, the gradients
 $\nabla_xL_c$ and $\nabla_{\lambda} L_c$ are respectively written as
\begin{align}
 \nabla_{x} {L}_{c}(x,\lambda) & = \nabla_x f(x) +c  E^\top( E x + \lambda/c -\prox{\phic}( E x + \lambda/c)), \label{gradL1}\\
  \nabla_{\lambda}{L}_{c}(x,\lambda)& = E x - \prox{\phic}( E x  + \lambda/c). \nonumber %\label{gradL2}
\end{align}
\end{itemize}
\end{proposition}
where $G\in \partial_C\prox{\phic}(z)$ at $z=Ex+\lambda/c$.
\subsection{The proximal point algorithm}
\subsubsection{Maximal monotone operator}
A set-valued operator $T\colon H \rightarrow 2^H$, defined on a real Hilbert space with an inner product $( \cdot,\cdot)$, is called \textit{monotone} if $( z - z^\prime, w-w^\prime) \ge 0$ for all $w \in T(z),\; w^\prime \in T(z^\prime)$. In addition, if the graph of $T$, $G(T) = \{(z,w)\in H\times H \mid w \in T(z)\}$,
is not properly contained in the graph of any other monotone operator,
then we call $T$ a \textit{maximal monotone} operator.
The fundamental property of the maximal monotone operator is that
for any $c>0$ and any $z\in H$, there exists a unique element $u$ such that
\begin{equation*}
z\in (I+cT)(u).
\end{equation*}
Thus, the operator $P:=(I+cT)^{-1}$ is a single-valued mapping defined on all of $H$.
$P$ is called the \textit{proximal mapping} associated with $cT$. It is nonexpansive
\begin{equation}\label{nonexpansive}
	\Vert P(z) - P(w)\Vert \le \Vert z - w\Vert, \quad \forall z, \forall w \in H;
\end{equation}
and $z=P(z)$ if and only if $0\in T(z)$.
The operator $(I +cT)^{-1}$ is also called the \textit{resolvent
  operator} of $T$ in the field of  functional analysis, and it is often denoted by $J_{cT}$.
Throughout this article, we call $P$ the proximal mapping and the point $P(z)$ for $z\in H$ \textit{the proximal point}.

\subsubsection{Proximal point algorithm}
Many problems can be recast as finding an element satisfying $0\in T(z)$ of a maximal monotone operator $T$.
A fundamental algorithm for solving the inclusion is the proximal point algorithm \cite{Rockafellar-Monooperproxpoin:76a}, which is a fixed-point algorithm
generating, for any initial point $z_0$ and a sequence $\{c_k\}$ of positive numbers bounded away from zero (i.e., $\inf_k c_k > 0$),
a sequence $\{z_k\}$ by the iteration
\begin{align*}
z_{k+1}= P_k(z_k):=(I + c_kT)^{-1}(z_k).
\end{align*}

The evaluation of $P_k(z)$ for a given $z$ is often nontrivial. In our
setting, the evaluation of $P_k(z)$ is equivalent to solving
a nonsmooth convex optimization problem (see Corollary~\ref{cor:P}).
In practice, it is often hopeless to compute the solution without computation error, and only an approximation is possible.
The following is the proximal point algorithm with the inexact evaluation of $P_k(z_k)$.
\renewcommand{\thealgorithm}{}
\begin{algorithm}
\caption{The proximal point algorithm} %\label{alg:0}
\begin{algorithmic}[1]
\STATE Let $r\ge 0$. Choose sequences of positive number $\{c_k\}_k$
such that $\liminf_{k\to \infty} c_k > 0$ and of $\{\epsilon_k\}_k$ such that $\sum_{k=0}^\infty \epsilon_k <\infty$.
\STATE Choose $z_0 \in H$ arbitrary.
\STATE Compute an approximation $z_{k+1}$ satisfying
\begin{equation}\label{ppa_stop1}
\Vert z_{k+1} - P_k(z_k) \Vert \le \epsilon_k\min(1,\Vert z_{k+1} -z_k\Vert^r). \tag{$\textrm{A}_r$}
\end{equation}
\STATE Set $z_k\leftarrow z_{k+1}$, $c_k\leftarrow c_{k+1}$ and $\epsilon_k\leftarrow \epsilon_{k+1}$, and return to Step 3.
\end{algorithmic}
\end{algorithm}
\renewcommand{\thealgorithm}{1}
\begin{remark}
At first glance, it appears meaningless to consider the criterion \eqref{ppa_stop1} because the criterion involves the unaccessible quantity $P_k(z_k)$.
The criterion is simply what is required in the context of the theoretical justification of the convergence of the iteration.
  In practical implementations of the algorithm, we will use an
  alternative criterion that does not involve $P_k(z_k)$ and that is sufficient for the condition \eqref{ppa_stop1}.
A computationally amenable criterion will be presented in the next section.
\end{remark}
We cite the general result on the convergence of the proximal point algorithm from \cite[Thm.~1]{Rockafellar-Monooperproxpoin:76a}.
\begin{theorem}[{\cite[Thm.~1]{Rockafellar-Monooperproxpoin:76a}}]\label{thm:mono_converge} Let $T$ be a maximal monotone operator.
Let $\{z_k\}$ be any sequence generated by the proximal point algorithm
satisfying the criterion \eqref{ppa_stop1} ($r\ge 0$) with $\liminf_{k\to\infty}c_k>0$.
Suppose that there exists at least one solution to $0\in T(z)$. Then, $\{z_k\}$ converges weakly to a point $z_\infty$ satisfying $0\in T(z_\infty)$.
\end{theorem}
In this article, we do not consider the rate of convergence of the proximal point algorithm. We refer the reader to the article, e.g.,
\cite{Luque-Asymconvanalprox:84}, for details on the convergence rates.
\section{The proximal method of multipliers}
In this section, we investigate the proximal method of multipliers (Algorithm~\ref{alg:1}).
\begin{algorithm}
\caption{The proximal method of multipliers}\label{alg:1}
\begin{algorithmic}[1]
\STATE Choose $r\ge0$ and $\{c_k\}_k$ such that $\liminf_{k\to\infty}\{c_k\}_k > 0$ and $\{\epsilon_k\}_k$ such that $\sum_{k=0}^\infty \epsilon_k <\infty$.
\STATE Choose $(x^0,\lambda^0)\in\R^n\times\R^m$.
\STATE Compute an approximation $x_{k+1}$ to the minimization problem
\begin{align}
\min_{x} \AL{x}{\lambda_k}{c_k} + 1/(2c_k)\Vert x-x_k\Vert^2,  \label{ppm_ag2}
\end{align}
with the accuracy
\begin{equation}\label{ppm_stop1}
\Vert \nabla_x \AL{x_{k+1}}{\lambda_k}{c_k} + c^{-1}_k(x_{k+1}-x_k)\Vert \le \frac{\epsilon_k}{c_k}\min(1,\Vert (x_{k+1},\lambda_{k+1}) -  (x_{k},\lambda_{k})\Vert^r ).
\end{equation}
where $\lambda_{k+1}$ is given by
\begin{align}
 \lambda_{k+1} &= \lambda_k + c_k(Ex_{k+1}-\mbox{prox}_{\phi/c_k}(Ex_{k+1}+\lambda_k/c_k)).\label{ppm_ag2_2}
\end{align}
\STATE Set $x_k\leftarrow x_{k+1}$, $\lambda_k\leftarrow \lambda_{k+1}$, $c_k\leftarrow c_{k+1}$ and $\epsilon_k\leftarrow \epsilon_{k+1}$ and return to Step 3.
\end{algorithmic}
\end{algorithm}
We begin with the statement of our main theorem on the convergence of the proximal method of multipliers.
\begin{theorem}\label{thm:ppm}
From any initial point $(x_0,\lambda_0)$, let $\{(x_k,\lambda_k)\}_k$ be any sequence generated by the proximal method of multipliers (Algorithm~\ref{alg:1}).
Suppose that $($A1$)$ and $($A2$)$ hold.
Then, $\{x_k\}_k$ converges to an optimal solution of \eqref{min}, and $\{\lambda_k\}_k$ converges to an optimal solution of \eqref{max}.
\end{theorem}
The proof follows Rockafellar's original idea in \cite{Rockafellar-Monooperproxpoin:76a} developed for the proximal method of multipliers for the inequality-constrained problem \eqref{min_ineq}.
We begin with the definition of a set-valued operator $T_L$ associated with the ordinary Lagrangian $L$: for a point $(x,\lambda)$, $T_L(x,\lambda)$ is defined as a set of all $(u,v)$ such that $u$ is a subgradient of convex function $L(\cdot,\lambda)$ at $x$ and $v$ is a subgradient of the convex function $-L(x,\cdot)$ at $\lambda$
\begin{equation}\label{TL}
T_L(x,\lambda):= \{(u,v) \mid u\in \partial_x L(x,\lambda), v \in \partial_\lambda(-L(x,\lambda))\},\quad \forall (x,\lambda)\in \R^n\times\R^m.
\end{equation}
We have the following Lemma.
\begin{lemma}%\label{lem:TL}
The map $T_L$ defined by \eqref{TL} is maximal monotone.
\end{lemma}
\begin{proof}
See  \cite[Cor.~2, p.~249]{Rockafellar-Monooperassowith:70}.
\end{proof}
Since $T_L$ is maximal, the operator $P_k=(I+c_kT_L)^{-1}$
is well defined as the single-valued operator on all $\R^n\times\R^m$.　
The next result provides the connection between the approximation $x_{k+1}$ and the proximal point $P_k(x_k,\lambda_k)$.
\begin{proposition}\label{lem_P}
Let $(x_k,\lambda_k)$ be a current iteration. Let $x_{k+1}$ be any approximation to the minimization problem \eqref{ppm_ag2} and
$\lambda_{k+1}$ be defined as \eqref{ppm_ag2_2}. Then, we have the estimate
	\begin{equation}\label{estimates1}
	\Vert (x_{k+1},\lambda_{k+1}) -P_k(x_k,\lambda_k)\Vert \le c_k\Vert \nabla_x \AL{x_k}{\lambda_k}{c_k} + c_k^{-1}(x_{k+1} -x_k)\Vert.
	\end{equation}
Moreover, if $x_{k+1}$ satisfies the condition \eqref{ppm_stop1}, we have
\begin{equation}\label{estimates2}
  \Vert (x_{k+1},\lambda_{k+1}) - P_k(x_k,\lambda_k)\Vert \le \epsilon_k\min(1,\Vert (x_{k+1},\lambda_{k+1}) -  (x_{k},\lambda_{k})\Vert^r ).
\end{equation}
\end{proposition}
\begin{proof}
The proof is adapted from \cite[Prop.~8]{Rockafellar-Monooperproxpoin:76a}, where the inequality-constrained optimization problem \eqref{min_ineq} was treated.

Let $w$ be $w= \nabla_x \AL{x_{k+1}}{\lambda_k}{c_k}+ c_k^{-1}(x_{k+1} - x_k)$. We show that
\begin{equation}\label{approx_P}
(x_{k+1},\lambda_{k+1}) = P_k(c_kw+x_k,\lambda_k).
\end{equation}
From \eqref{gradL1} and the definition of $\lambda_{k+1}$ and the ordinary Lagrangian $L$, it follows that
	\begin{equation*}
	\nabla_x\AL{x_{k+1}}{\lambda_k}{c_k} = \nabla_x f(x_{k+1}) + E^{\top}\lambda_{k+1} = \partial_x L(x_{k+1},\lambda_{k+1}).
	\end{equation*}
Consequently, we have
\begin{equation}\label{x}
  w + c^{-1}(x_k-x_{k+1}) = \partial_x L(x_{k+1},\lambda_{k+1}).
\end{equation}	
However, since $\lambda_{k+1}=\lambda_k+c_k(Ex_{k+1}-\prox{\phic}(Ex_{k+1}+\lambda_k/c_k))=\prox{c_k\phi^\ast}(\lambda_k + c_kEx_{k+1})$,
it follows that  $\lambda_{k+1}$ is the solution of the minimization problem
	\begin{equation*}
	\lambda_{k+1} = \arg  \min_{\mu} \left(\frac{1}{2c_k}\Vert \mu - (\lambda_k+c_k Ex_{k+1})\Vert^2  + \phi^\ast(\mu)\right),
	\end{equation*}
whose optimality condition is given by
\begin{equation*}
0 \in \partial \phi^\ast(\lambda_{k+1}) + c_k^{-1}(\lambda_{k+1} - (c_kEx_{k+1}+\lambda_k)).
\end{equation*}
By the definition of $L(x,\lambda)$, this is equivalent to
	\begin{align}\label{lambda}
c^{-1}_k(\lambda_{k} -\lambda_{k+1})\in \partial_\lambda(- L)(x_{k+1},\lambda_{k+1}).
	\end{align}
	From \eqref{x} and \eqref{lambda}, we have
$(w + \frac{x_k - x_{k+1}}{c_k},\frac{\lambda_k-\lambda_{k+1}}{c_k})   \in   T_L(x_{k+1},\lambda_{k+1})$,
which is equivalent to
$(c_kw + x_{k},\lambda_k)\in (I + c_kT_L)(x_{k+1},\lambda_{k+1})$.
Thus, we have $P_k(c_kw+x_k,\lambda_k)=(x_{k+1},\lambda_{k+1})$. % by definition of $P_k$.
This proves the claim \eqref{approx_P}.

Since $P_k$ is nonexpansive (cf. \eqref{nonexpansive}), we have
	\begin{align*}
	\Vert (x_{k+1},\lambda_{k+1})-P_k(x_k,\lambda_k)\Vert& = \Vert P_k(c_kw+x_k,\lambda_k)-P_k(x_k,\lambda_k)\Vert \\
& \le \Vert (c_kw+x_k,\lambda_k)-(x_k,\lambda_k)\Vert = c_k\Vert w\Vert.
	\end{align*}

Finally, the estimate \eqref{estimates2} readily follows from \eqref{ppm_stop1} and \eqref{estimates1}.
\end{proof}
We are now in the position to prove Theorem~\ref{thm:ppm}.
\begin{proof}[Proof of Theorem~\ref{thm:ppm}]
	We first show that there exists at least one solution to $0\in T_L(x,\lambda)$.
	From (A1) and (A2) in section~\ref{sec:dual_Lagrangian}, it follows that there exist optimal solutions for both \eqref{min} and \eqref{max}.
	By virtue of Proposition \ref{prop:saddle}, a pair of optimal solutions is a saddle point of $L$ and vice versa.
	Clearly, $(x,\lambda)$ is a saddle point of $L$ if and only if $0\in \partial_x L(x,\lambda)$ and $0\in \partial_\lambda (-L)(x,\lambda)$.
	The inclusions are equivalently written as $0\in T_L(x,\lambda)$.
	Hence, the assertion is true.
	
	In the above argument, we have also proven that the set of solutions to $0\in T_L(x,\lambda)$ is identical to the set of pairs of optimal solutions for \eqref{min} and \eqref{max}, which is also identical to the set of saddle points of $L$. Hence, it is sufficient to show that the sequence converges to a solution to $0\in T_L(x,\lambda)$.
	However, since $x_{k+1}$ satisfies the estimates \eqref{ppm_stop1}, it follows from Proposition~\ref{lem_P} that
	\begin{equation*}
	\Vert (x_{k+1},\lambda_{k+1}) - P_k(x_k,\lambda_k)\Vert \le \epsilon_k\min(1,\Vert (x_{k+1},\lambda_{k+1}) -  (x_{k},\lambda_{k})\Vert^r ),
	\end{equation*}
	and thus, by Theorem~\ref{thm:mono_converge}, we can conclude that
	the sequence $\{(x_k,\lambda_k)\}_k$ converges to a solution to $0\in T_L(x,\lambda)$.
	This completes the proof.
\end{proof}
At the end of this section, we provide two interpretations of the proximal point $P_k(x_k,\lambda_k)$.
\begin{corollary}\label{cor:P}
The following assertions are valid:
	\begin{enumerate}
		\item Let us denote $(\bar{x}_k,\bar{\lambda}_k) = P_k(x_k,\lambda_k)$. Then, $\bar{x}_k$ is a unique minimizer of \eqref{ppm_ag2} and $\bar{\lambda}_k = \lambda_k + c_k(E\bar{x}_k -\prox{\phi/c_k}(E\bar{x}_k+\lambda_k/c_k))$.
		\item  Let us define the regularized Lagrangian $\mathcal{L}_k$ by
		\begin{equation*}%\label{reg_lag}
		\mathcal{L}_k(x,\lambda):=L(x,\lambda) + 1/(2c_k)\Vert x -x_k\Vert^2 - 1/(2c_k)\Vert \lambda - \lambda_k\Vert^2.
		\end{equation*}
		Then, $P_k(x_k,\lambda_k)$ is the unique saddle point of $\mathcal{L}_k$.
		\end{enumerate}
\end{corollary}
\begin{proof}
From Proposition~\ref{prop:KKT}, we know that the augmented Lagrangian $L_{c_k}(x,\lambda_k)$ is convex with respect to $x$.
Since $1/(2c_k)\Vert x - x_k\Vert^2$ is strictly convex, the function $L_{c_k}(x,\lambda_k)+1/(2c_k)\Vert x - x_k\Vert^2$ is strictly convex and thus
has a unique minimizer.  Let $x^\ast_k$ be the unique minimizer and $\lambda^\ast_k=\lambda_k+c_k(Ex^\ast_k-\prox{\phi/c_k}(Ex^\ast_k+\lambda_k/c_k))$. From the optimality condition, it follows that
\begin{equation*}
\nabla_x \AL{x^\ast_k}{\lambda_k}{c_k}+ c_k^{-1}(x^\ast_{k} - x_k)=0.
\end{equation*}
However, from \eqref{estimates1}, it follows that $(x^\ast_k,\lambda^\ast_k) = P_k(x_k,\lambda_k)$, which proves the first assertion.

The second assertion is valid from the following equivalences: By definition, $(\bar{x}_k,\bar{\lambda}_k) = (I+c_k T_L)^{-1}(x_k,\lambda_k)$. This is equivalently written as $ \left(c^{-1}_k({x_k-\bar{x}_k}) , c^{-1}_k(\lambda_k-\bar{\lambda}_k) \right)\in T_L(\bar{x}_k,\bar{\lambda}_k)$, which is also written as
\begin{align*}
 &  \left\{\begin{array}{l}
   0\in  \partial_x (L(x,\bar{\lambda}_k) + 1/(2c_k)\Vert x - x_k\Vert^2 )|_{x=\bar{x}_k}, \\[5pt]
   %\quad\mbox{and}\quad
   0\in  \partial_\lambda (-L(\bar{x}_k, \lambda) + 1/(2c_k)\Vert \lambda -\lambda_k\Vert^2)|_{\lambda=\bar{\lambda}_k}. %\\[5pt]
   \end{array}\right.
 \end{align*}
 This means that $\bar{x}_k$ is the stationary point of $\mathcal{L}_k(\cdot,\bar{\lambda}_k)$ and
  $\bar{\lambda}_k$ is that of $\mathcal{L}_k(\bar{x}_k,\cdot)$; in other words,
  $0\in  \partial_x \mathcal{L}_k(\bar{x}_k,\bar{\lambda}_k)$
  and
  $0\in  \partial_\lambda (\mathcal{-L}_k)(\bar{x}_k,\bar{\lambda}_k)$.
  Since $\mathcal{L}_k(x,\bar{\lambda}_k)$ and $\mathcal{-L}_k(\bar{x}_k,\lambda)$ are strictly convex with respect to $x$ and $\lambda$, respectively, it follows that
\begin{align*}
   \left\{\begin{array}{rl}
   \mathcal{L}_k(x,\bar{\lambda}_k)\ge   \mathcal{L}_k(\bar{x}_k,\bar{\lambda}_k) + ( 0, x-\bar{x}_k) & \forall x\in \R^n,
  \\
   -\mathcal{L}_k(\bar{x}_k,\lambda)\ge   -\mathcal{L}_k(\bar{x}_k,\bar{\lambda}_k) + ( 0, \lambda-\bar{\lambda}_k) & \forall \lambda\in \R^m,
    \end{array}\right.
\end{align*}
where the equalities hold if and only if $x=\bar{x}_k$ and $\lambda=\bar{\lambda}_k$. Consequently, we have
%\begin{align*}
$\mathcal{L}_k(\bar{x}_k,\lambda)   \le   \mathcal{L}_k(\bar{x}_k,\bar{\lambda}_k) \le  \mathcal{L}_k(x,\bar{\lambda}_k)$,
for all $(x,\lambda)\in \R^n\times \R^m$.
%\end{align*}
This implies that the proximal point $(\bar{x}_k,\bar{\lambda}_k)$ is the unique saddle point of the regularized Lagrangian $\mathcal{L}_k$.
\end{proof}
We readily see from Corollary~\ref{cor:P} that the proximal method of multipliers with exact minimization is the realization of the proximal point algorithm with exact evaluation of $P_k(x_k,\lambda_k)$.
\begin{theorem}
The proximal point algorithm with exact minimization
 \begin{equation*}
 (x_{k+1},\lambda_{k+1})=P_k(x_k,\lambda_k)
  \end{equation*}
 is equivalent to the proximal method of multipliers with exact minimization
\begin{align*}
\left\{
\begin{array}{rcl}
x_{k+1}& = &\displaystyle \arg \min_{x} \AL{x}{\lambda_k}{c_k} + 1/(2c_k)\Vert x-x_k\Vert^2, \\[5pt]
\lambda_{k+1} &= &\displaystyle \lambda_k + c_k(Ex_{k+1}-\prox{\phi/c_k}(Ex_{k+1}+\lambda_k/c_k)).
\end{array}\right.
\end{align*}
\end{theorem}
\section{Newton method for solving the subproblem}
 In the previous section, we established the convergence of the proximal method of multipliers, in which we assumed that
an approximation $x_{k+1}$ to the inner minimization problem \eqref{ppm_ag2} is available.
In this section, we develop a numerical method for computing $x_{k+1}$ satisfying the estimate \eqref{ppm_stop1}.
Let us recall the problem \eqref{ppm_ag2}:
\begin{equation}\label{min_inner}
  \min_{\xi}\psi(\xi):=\AL{\xi}{\lambda}{c} + 1/(2c)\Vert \xi - x\Vert^2.
\end{equation}
We use the symbol $\xi$ to denote the variable of $\psi$ to be minimized, and we drop the subscript $k$ from $x_k,\lambda_k$ and $c_k$ for simplicity.
Let $\xi^\ast$ denote the first $n$ components of the proximal point $P_k(x,\lambda)$. Recall that $\xi^\ast$ is the unique minimizer of \eqref{min_inner} (see Corollary~\ref{cor:P}), i.e., $\xi^\ast = \arg\min_\xi \psi(\xi)$.  Apparently, the necessary and sufficient optimality condition for $\xi^\ast$ to be
the minimizer for $\psi$ is given as
\begin{equation}\label{equ:opt}
\nabla \psi(\xi^\ast )=0.
\end{equation}
Now, we describe an algorithm for solving \eqref{equ:opt} and state
results on the convergence properties. For this purpose, we need the
concepts of the generalized Jacobian and semismoothness. \\
\textbf{Generalized Jacobian.} Let $\Phi \colon \R^m\rightarrow \R^n$ be a locally Lipschitz continuous map. Rademacher's
Theorem \cite[Sect.~3.1.2]{Evans+Gariepy-Meastheofineprop:92} states that a locally continuous map is
differentiable almost everywhere. Denote by $N_{\Phi}$ a set of measure zero such that $\Phi$ is differentiable
on $\R^m\setminus N_{\Phi}$. The \textit{limiting Jacobian} of $\Phi$ at $\xi$ is the set
\begin{equation*}
\partial_B \Phi(\xi) := \left\{G \in \R^{n\times m} \mid  \exists \{\xi^k\} \subset \R^m\setminus N_{\Phi} \mbox{ with } \xi^k\rightarrow \xi,
D_\xi\Phi(\xi^k) \rightarrow G\right\}.
\end{equation*}
\textit{(Clarke's) generalized Jacobian} $\partial\Phi(\xi)$ of $\Phi$ at $\xi \in\R^m$ is the convex hull of the limiting Jacobian:
\begin{equation*}
\partial \Phi(\xi) = \mbox{conv}(\partial_B\Phi(\xi)).
\end{equation*}
We denote by $\partial_B \Phi$ the set-valued map $\xi \rightarrow \partial_B\Phi(\xi)$ for $\xi\in\R^m$. The set-valued map $\partial\Phi$ for the generalized Jacobian is defined analogously. \\
\textbf{Semismoothness.} Let $\Phi: \R^n \rightarrow \R^m$ be a locally Lipschitz continuous map.
We say that $\Phi$ is semismooth at a point $\bar{x}$ if $\Phi$ is directionally differentiable near $\bar{x}\in \R^n$ and
\begin{equation*}
 \lim_{\substack{\bar{x} \neq x\to  \bar{x} \\  G\in \partial \Phi(x)}} \frac{\Vert \Phi(x)+G(\bar{x} - x) - \Phi(\bar{x})\Vert}{\Vert x  - \bar{x}\Vert} = 0.
\end{equation*}
If the above requirement is strengthened to
\begin{equation*}
 \limsup_{\substack{\bar{x} \neq x\to  \bar{x} \\  G\in \partial \Phi(x)}} \frac{\Vert \Phi(x)+G(\bar{x} - x) - \Phi(\bar{x})\Vert}{\Vert x  - \bar{x}\Vert^2} <\infty.
\end{equation*}
we say that $\Phi$ is strongly semismooth at $\bar{x}$. If $\Phi$ is (strongly) semismooth at each point of a subset $\Omega\subset \R^n$, we say that $\Phi$ is (strongly) semismooth on $\Omega$.

We now consider the generalized Jacobian of the map $\nabla\psi$. Here, we recall that
\begin{equation*}
\nabla\psi(\xi) = \nabla f(\xi) +c  E^{\top}( E \xi + \lambda/c -\prox{\phi/c}( E \xi + \lambda/c))+(\xi-x)/c.
\end{equation*}
 From the chain rule of the generalized Jacobian \cite[Thm.~4]{Imbert-SuppfuncClargene:02}, it follows that
\begin{equation*}
  \partial(E^{\top}\prox{\phi/c}(E\xi+\lambda/c))
 \subset \{E^{\top} G E \mid G\in \partial (\prox{\phi/c})(E\xi+\lambda/c)\},
\end{equation*}
and thus,
\begin{align}\label{def:T}
\partial(\nabla \psi)(\xi)
\subset T(\xi): = \{ \nabla^2 f(\xi) + c^{-1} I + c E^{\top}(I - G)E \mid G\in \partial (\prox{\phi/{c}})(E\xi+\lambda/c)\}.
\end{align}
With these settings, we formulate a Newton-type method with line search (Algorithm~\ref{alg:2}). % with the stopping criteria \eqref{estimates1}.
\renewcommand{\thealgorithm}{2}
\begin{algorithm}
\caption{Newton method with line search for the inner problem \eqref{min_inner}}\label{alg:2}
\begin{algorithmic}[1]
\STATE Input: $x\in \R^n$, $\lambda\in\R^m$, $c>0$, $\gamma\in(0,1/2)$, $\rho\in(0,1)$.
\STATE Initialize $\xi_0 \in\R^n$ (e.g., $\xi_0 = x$).
\STATE Select $V_\ell \in T(\xi_\ell)$ and find $d_\ell$ satisfying
\begin{equation*}%\label{Newton}
V_\ell d_\ell = -\nabla\psi(\xi_\ell).
\end{equation*}
\STATE Find the smallest nonnegative integer $i$ such that
\begin{equation}\label{arjimo}
  \psi(\xi_\ell+\rho^i d_\ell)\le \psi(\xi_\ell) + \gamma\rho^i \nabla \psi(\xi_\ell)^{\top} d_\ell,
\end{equation}
and set $\tau_\ell = \rho^i$.
\STATE Set $\xi_{\ell+1} = \xi_{\ell}+\tau_\ell d_\ell$.
\STATE
Set $\ell \leftarrow \ell+1$ and return to Step 3.
\end{algorithmic}
\end{algorithm}

The implementation of Newton's method requires the explicit representation of the Jacobians of the proximal mapping $\prox{\phic}(z)$.
We refer the reader to the article \cite{Patrinos+StellaETAL-ForwtrunNewtmeth:14} for the closed forms of a variety of proximal mappings
that frequently appear in practical applications.

The next theorem states the convergence results of Algorithm~\ref{alg:2}.
This theorem guarantees that an approximation $x_{k+1}$ to $P_k(x_k,\lambda_k)$ satisfying the criterion \eqref{ppm_stop1} is obtained within a finite number of iterations.
\begin{theorem}\label{thm:converge_alg2} Let $\xi^\ast$ be the global minimizer of $\psi$.
Let $\{\xi_\ell\}$ be a sequence of iterates by Algorithm~\ref{alg:2}. The following statements are valid.
\begin{enumerate}\renewcommand{\labelenumi}{(\roman{enumi})}
\item The sequence globally converges to $\xi^\ast$.
\item If $\prox{\phi/c}$ is semismooth at  $\xi^\ast$, then a unit step size $\tau_\ell = 1$ in the Armijo rule
  is eventually accepted, i.e., there exists $\ell_0$ such that $\tau_\ell =1 $ for all $\ell\ge \ell_0$.
\item Under the above assumptions, the convergence rate is Q-superlinear.
In addition, if
$\prox{\phi/c}$ is strongly semismooth at $\xi^\ast$, then the convergence rate is Q-quadratic.
\end{enumerate}
\end{theorem}
\begin{remark}
Theorem~8.3.19 in \cite[Ch.~8]{Facchinei+Pang-Finivariineqcomp:03} considers a similar algorithm, where $V_\ell\in\partial_B(\nabla \psi)(\xi)$ is assumed. Note that $\partial_B(\nabla\psi)(\xi) \subset T(\xi)$ is always true but $\partial_B(\nabla\psi)(\xi) = T(\xi)$ does not necessarily hold, which implies that $V_\ell \in T(\xi)$ may not be a member of $\partial_B(\nabla\psi)(\xi)$. For this reason,
Theorem~8.3.9 cannot be directly applied to prove the convergence of Algorithm~\ref{alg:2}. We need the concept of \textit{linear Newton approximation} (LNA for short) and the convergence result of a \textit{linear Newton method}, which will be discussed below. For a comprehensive treatment and for further references on these subjects, one may refer to \cite[Ch.~7]{Facchinei+Pang-Finivariineqcomp:03}.
\end{remark}
\begin{remark}
The Newton-type method (Algorithm~\ref{alg:2}) was implemented in  \cite{Tomioka+SuzukiETAL-SupeConvDualAugm:11} for solving subproblems \eqref{step1_dal} of the method of multipliers (the augmented Lagrangian method by Fortin). However, the article did not investigate the global convergence or the convergence rate of the Newton-type method.
\end{remark}
\subsection{Convergence of the Newton method} %Algorithm~\ref{alg:2}}
In this section, we provide a proof of Theorem~\ref{thm:converge_alg2}.
Throughout this section, we use the following notations: for a given sequence $\{a_\ell\}_{\ell}$, we denote a subsequence by $\{a_{\ell}\}_{\ell \in \chi}$, where $\chi$ is a subset of $\mathbb{N}$. Let $A\in\R^{n,n}$ and $B\in\R^{n,n}$ be symmetric positive semidefinite matrices. We denote $A\preceq B$ if $(d,Ad)\le (d,Bd)$ for all $d\in \R^n$.
\subsubsection{The proof of Theorem~\ref{thm:converge_alg2} (i)}
To show the global convergence of Algorithm~\ref{alg:2}, we need the result from \cite[Thm.~3.2]{Patrinos+StellaETAL-ForwtrunNewtmeth:14} that states the basic property of the generalized Jacobian of the proximal mapping.
\begin{lemma}[{\cite[Thm.~3.2]{Patrinos+StellaETAL-ForwtrunNewtmeth:14}}]\label{lem:jacobian}
	For any $\phi \in \Gamma_0(\R^m)$, every $G \in \partial(\prox{\phic})(z)$ is a symmetric positive semidefinite matrix with
	$\Vert G\Vert \le 1$.
\end{lemma}
Using the result of Lemma~\ref{lem:jacobian}, one can readily show that $V\in T(\xi)$ is symmetric and strictly positive definite:
\begin{lemma}\label{lem:V}
Let $\xi\in\R^n$, and let $T(\xi)$ be the set of matrices defined by \eqref{def:T}.
Every $V\in T(\xi)$ is a symmetric positive semidefinite matrix with
\begin{equation*}
c^{-1}I \preceq V  \preceq   \nabla^2 f(\xi) + (c^{-1}+c\Vert E\Vert^2) I.
\end{equation*}
\end{lemma}
\begin{proof}
We recall that $V\in T(\xi)$ is given as
\begin{equation*}
V= \nabla^2 f(\xi) + c^{-1}I + cE^{\top}(I-G)E,\quad G\in \partial(\prox{\phi/c})(E\xi + \lambda/c).
\end{equation*}
The first relation $c^{-1}I \preceq V$ is obvious.
For the second relation, by Lemma~\ref{lem:jacobian}, we have $(d,E^T(I-G)Ed) =  (Ed,(I-G)Ed)\le  (Ed,Ed)\le \Vert E\Vert^2\Vert d\Vert^2$ for $d\in\R^n$. The rest is obvious, and we complete the proof.
\end{proof}
\begin{proof}[Proof of Theorem~\ref{thm:converge_alg2} (i)]
We show that $\{\xi_\ell\}$ converges to the global minimizer $\xi^\ast$ of $\psi$.
Let $\xi_\ell$ be a nonstationary point of $\psi$, i.e., $\nabla\psi(\xi_\ell)\neq0$. Then, from Lemma~\ref{lem:V}, we have $d_\ell = -V_\ell^{-1}\nabla \psi_k(\xi_\ell)\neq 0$ and
$\nabla\psi(\xi_\ell)^{\top}d_\ell  = - (d_\ell, V_\ell d_\ell) = - \Vert V^{1/2}_{\ell}d_\ell\Vert^2<0$, which justifies that there exists $\tau_\ell=\rho^i$ that satisfies \eqref{arjimo}.
Since $\psi(\xi_{\ell+1})\le \psi(\xi_{\ell}) + \gamma \tau_{\ell}\nabla \psi(\xi_\ell)^{\top} d_\ell<\psi(\xi_{\ell})$ and $0\le \psi(\xi_\ell)$,
we see that every $\xi_\ell$ belongs to the set $\{\xi \in \R^n \mid
0\le \psi(\xi) \le\psi(\xi_0)\}$, which is a bounded subset of $\R^n$ because $\psi$ is coercive.
Namely, the generated sequence is $\{\xi_\ell\}_{\ell}$ bounded. Therefore, there exists an accumulation point $\bar{\xi}\in\R^n$ of $\{\xi_\ell\}_{\ell}$ and a subsequence $\{\xi_\ell\}_{\ell\in \chi}$ that converges to $\bar{\xi}$.
Let $\delta>0$ be an arbitrary number, and let $B_\ell = B(\xi_\ell, \delta)$ be a ball with center $\xi_\ell$ and radius $\delta$. Then, there exists a finite subset $J\subset \mathbb{N}$ such that $\cup_{\ell}B_\ell = \cup_{\ell\in J}B_\ell$.
Thus, $L = \sup_{\ell\in J}L_\ell$ is finite, and consequently, we have $\nabla^2 f(\xi_\ell)  \preceq  L I$ for all $\ell$.
By applying \cite[Prop.~8.3.7]{Facchinei+Pang-Finivariineqcomp:03} to the subsequence $\{\xi_\ell\}_{\ell\in\chi}$, we obtain that $\bar{\xi}$ is a stationary point, which is also the global minimizer of \eqref{min_inner}, i.e., $\bar{\xi}=\xi^\ast$. Since every accumulation point of $\{\xi_\ell\}_{\ell}$ coincides with the global minimizer $\xi^\ast$, the entire sequence converges to $\xi^\ast$.
\end{proof}
\subsubsection{The proof of Theorem~\ref{thm:converge_alg2} (ii)}
For a given convergent sequence $\{x_k\}$ with limit $x^\ast$, we say that a sequence $\{d_k\}$ is superlinearly convergent with respect to $\{x_k\}$ if the following limit holds:
\begin{equation*}
\lim_{k\to\infty}\frac{\Vert x_k+d_k-x^\ast\Vert}{\Vert x_k-x^\ast\Vert} = 0.
\end{equation*}

We state the result relevant to the step size in the Armijo rule. The following result is borrowed from
\cite[Prop.~8.3.13]{Facchinei+Pang-Finivariineqcomp:03}.%, which we slightly modify according to our needs.
\begin{proposition}[{\cite[Prop.~8.3.18]{Facchinei+Pang-Finivariineqcomp:03}}]\label{prop:13}
Let $\theta : \R^n\rightarrow \R$ be a continuously differentiable function with $\nabla \theta$ semismooth near a zero $\eta^\ast$ of $\nabla \theta$.
Suppose that a sequence $\{\eta_\ell\}$ converges to $\eta^\ast$ with $\eta_\ell \neq \eta^\ast$ for all $\ell$.
Let $\{d_\ell\}$ be a superlinearly convergent sequence with respect to $\{\eta_\ell\}$.
If every matrix belonging to the generalized Jacobian $\partial(\nabla \theta)(\eta^\ast)$ is strictly positive definite, then
the following two statements hold: \\
(a) There exists a positive constant $\rho$ and $\ell_0$ such that
\begin{equation*}
\nabla\theta(\eta_\ell)^{\top} d_\ell\le -\rho \Vert d_\ell\Vert^2,\quad \mbox{for all}\quad \ell \ge \ell_0.
\end{equation*}
(b) For every $\gamma\in(0,1/2)$, there exists a positive constant $\ell^\prime_0$ such that
\begin{equation*}
  \theta(\eta_\ell + d_\ell)\le \theta(\eta_\ell)+\gamma \nabla\theta(\eta_\ell)^{\top} d_\ell,\quad \mbox{for all}\quad \ell \ge \ell^\prime_0.
\end{equation*}
\end{proposition}
We employ Proposition~\ref{prop:13} to prove that $\tau_\ell =1$ is accepted for all $\ell$ sufficiently large by the Armijo rule in Algorithm~\ref{alg:2}. Obviously, $\nabla \psi$ is semismooth. Since $\partial(\nabla \psi)(\xi^\ast)\subset T(\xi^\ast)$ by definition and every $V\in T(\xi^\ast)$ is nonsingular by Lemma~\ref{lem:V}, we know that the generalized Jacobian $\partial(\nabla \psi)(\xi^\ast)$ is strictly positive definite. It remains to show that $\{d_\ell\}$ is a superlinear convergent sequence with respect to $\{\xi_\ell\}$ generated by Algorithm~\ref{alg:2}, namely, we must show that
\begin{equation*}
\lim_{\ell\to\infty}\frac{\Vert \xi_\ell +d_\ell -\xi^\ast\Vert}{\Vert \xi_\ell - \xi^\ast\Vert}=0.
\end{equation*}
To this end, we exploit the theory of the linear Newton approximation scheme.
\begin{definition}\label{def:linear Newton}
	Let $\Phi\colon \R^n \rightarrow \R^n$ be locally Lipschitz continuous.
	If there exists a set-valued map $S\colon \R^n \rightrightarrows \R^{n\times n}$ such that:
	\begin{enumerate}
		\item[\textrm{(a)}] The set of matrices $S(\eta)$ is nonempty and compact for each $\eta\in\R^n$;
		\item[\textrm{(b)}] $S$ is upper semicontinuous at $\bar{\eta}$;
		\item[\textrm{(c)}] The following limit holds:
		\begin{equation*}
		\lim_{\substack{\bar{\eta}\neq \eta \rightarrow \bar{\eta}\\ H\in S(\eta) }}
		\frac{ \Vert \Phi(\eta) + H(\bar{\eta}-\eta) - \Phi(\bar{\eta})\Vert}{\Vert \eta -\bar{\eta}\Vert} = 0;
		\end{equation*}
	\end{enumerate}
	we say the map $\Phi$ admits a \textit{linear Newton approximation (LNA)} of $\Phi$ at $\bar{\eta}\in\R^n$. We also say that
	$S$ is a linear Newton approximation scheme of $\Phi$ at $\bar{\eta}\in\R^n$.
	%We also say that
	If \textrm{(c)} is strengthened to
	\begin{enumerate}
		\item[($\textrm{c}^{\prime}$)]
		\begin{equation*}
		\limsup_{\substack{\bar{\eta}\neq \eta \rightarrow \bar{\eta}\\ H\in S(\eta) }}
		\frac{ \Vert \Phi(\eta) + H(\bar{\eta}-\eta) - \Phi(\bar{\eta})\Vert}{\Vert \eta -\bar{\eta}\Vert^2} <\infty,
		\end{equation*}
	\end{enumerate}
	we say that $S$ is a strongly linear Newton approximation scheme at $\bar{\eta}$.
\end{definition}
\begin{lemma}\label{B_LNA}
Assume that a locally Lipschitz map $\Phi \colon \R^m \rightarrow \R^m$ is (strongly) semismooth at $\eta\in \R^m$; then,
each of $\partial \Phi$ and $\partial_B \Phi$ defines a (strong) LNA scheme of $\Phi$ at $\eta$.
\end{lemma}
\begin{proof}
It follows from \cite[Prop.~7.1.4]{Facchinei+Pang-Finivariineqcomp:03} that the set-valued map $\partial \Phi $ satisfies conditions {\rm (a)} and {\rm (b)} of Definition~\ref{def:linear Newton}, while from \cite[Thm.~7.4.3]{Facchinei+Pang-Finivariineqcomp:03}, the map satisfies condition {\rm (c)}. We refer the proof for the limiting Jacobian to \cite[Prop.~7.5.16]{Facchinei+Pang-Finivariineqcomp:03}.
\end{proof}
\begin{lemma}\label{prop:sum}
Let $\Phi_1 \colon \R^n\rightarrow \R^m$ and $\Phi_2 \colon \R^n \rightarrow \R^m$ be locally Lipschitz maps,
with $T_1$ and $T_2$ being LNAs of $\Phi_1$ and $\Phi_2$ at $\xi$, respectively. For linear maps, $A_1, A_2 \colon \R^{m}\rightarrow \R^\ell$, the map $A_1\Phi_1+A_2\Phi_2$ is locally Lipschitz continuous, and $A_1 T_1 + A_2 T_2$ is a LNA of the map.
\end{lemma}
\begin{proof}
The assertion directly follows from \cite[Cor.~7.5.18]{Facchinei+Pang-Finivariineqcomp:03}.
\end{proof}
\begin{lemma}[{\cite[Thm~7.5.17]{Facchinei+Pang-Finivariineqcomp:03}}]\label{prop:chain}
Let $\Phi \colon \R^{n}\rightarrow \R^m$ be a locally Lipschitz continuous map defined by $\Phi(\xi)=\Phi_1(\Phi_2(\xi))$, where $\Phi_1 \colon \R^m\rightarrow \R^m$ and $\Phi_2 \colon \R^n\rightarrow \R^m$ are both locally Lipschitz continuous. Suppose that $T_1$ and $T_2$ are (strong) linear Newton approximation schemes of $S_1$ and $S_2$ at $\Phi_2(\xi)$ and $\xi$, respectively. Then,
\begin{equation*}
  S(\xi)=\{G_1G_2\colon G_1\in T_1(\Phi_2(\xi)),\; G_2\in S_2(\xi)\},
\end{equation*}
is a (strong) LNA of $\Phi$ at $\xi\in \R^n$.
\end{lemma}
We now show that $T(\xi)$ is a (strong) LNA of $\nabla \psi(\xi)$.
\begin{proposition}\label{prop:LNA}
Let $\xi\in\R^n$ and assume that the proximal mapping $\prox{\phi/c}$ is (strongly) semismooth at $z=E\xi+\lambda/c$.
Then, $T(\xi)$ defined by \eqref{def:T} is a (strong) linear Newton approximation scheme of $\nabla\psi$ at $\xi$.
\end{proposition}
\begin{proof}
From Lemma~\ref{B_LNA}, Lemma~\ref{prop:sum} and Lemma~\ref{prop:chain}, we know that the map $\xi\rightarrow \prox{\phi/c}(E\xi+\lambda/c)$ admits  a (strong) LNA scheme $\{GE \mid G\in\partial(\prox{\phi/c})(E\xi+\lambda/c)\}$.

Let us write $\Phi(\xi)=\nabla f(\xi)+c^{-1}(\xi-x)+cE^{\top}E\xi$. It suffices to show that $\nabla\Phi(\xi)=\nabla^2f(\xi)+ c^{-1}I+cE^{\top}E$ is a strong LNA of $\Phi(\xi)$. Since $\nabla f(\xi)$ is locally Lipschitz continuous, we have (cf. \cite[Prop.~7.2.9]{Facchinei+Pang-Finivariineqcomp:03})
\begin{align*}
\Vert \Phi(\xi)-\Phi(\eta)-\nabla\Phi(\xi)(\xi-\eta)\Vert =\Vert \nabla f(\xi)- \nabla f(\eta)-\nabla^2f(\xi)(\xi-\eta)\Vert=O(\Vert \xi- \eta\Vert^2),
\end{align*}
which completes the proof.
\end{proof}
The following estimate is a key to establishing that the direction $\{d_\ell\}$ is superlinear convergent with respect to $\{\xi_\ell\}$.
\begin{lemma}\label{prop:d}
Let $\{\xi_\ell\}$, $\{d_\ell\}$ and $\{V_\ell\}$ be sequences generated by Algorithm~\ref{alg:2}.%, and $\{V_\ell\}$ be the sequence of matrices used in \eqref{Newton}.
We have the following estimate:
\begin{equation}\label{est}
  \Vert \xi_\ell + d_\ell - \xi^\ast \Vert \le c \Vert \nabla \psi(\xi_\ell)-\nabla \psi(\xi^\ast) - V_\ell(\xi_\ell - \xi^\ast)\Vert, \quad \mbox{for all } \ell.
\end{equation}
\end{lemma}
\begin{proof}
The estimate follows from the argument in the proof of \cite[Thm.~8.3.15]{Facchinei+Pang-Finivariineqcomp:03}. We provide the derivation of the estimate for completeness.
Adding $\nabla \psi(\xi_\ell)+V_\ell d_\ell = 0$ and $\nabla \psi(\xi^\ast)=0$ to $d_\ell = V_{\ell}^{-1}\nabla \psi(\xi_\ell)$ yields
%\begin{equation*}
 $V_\ell(\xi_\ell+d_\ell -\xi^\ast) + \nabla \psi(\xi_\ell)-\nabla \psi(\xi^\ast) -V_\ell(\xi_\ell-\xi^\ast) =0$.
%\end{equation*}
By multiplying $x_\ell+d_\ell -\xi^\ast$ and then by using $c^{-1}I \preceq V_\ell$ and the Cauchy-Schwarz inequality, we have the desired estimate \eqref{est}.
\end{proof}
\begin{proof}[Proof of Theorem~\ref{thm:converge_alg2} (ii)]
From Lemma~\ref{prop:d}, we have
\begin{equation*}
\frac{\Vert \xi_\ell +d_\ell -\xi^\ast\Vert}{\Vert \xi_\ell - \xi^\ast\Vert}\le \frac{c\Vert \nabla \psi(\xi_\ell)-\nabla \psi(\xi^\ast) - V_\ell(\xi_\ell - \xi^\ast)\Vert}{\Vert \xi_\ell - \xi^\ast\Vert}.
\end{equation*}
By Proposition~\ref{prop:LNA}, if the proximal mapping is semismooth at $\xi^\ast$, the left-hand side tends to 0 as $\ell\to\infty$ because $\lim_{\ell\to\infty}\xi_\ell = \xi^\ast$. Consequently,
$
\lim_{\ell\to\infty}\frac{\Vert \xi_\ell +d_\ell -\xi^\ast\Vert}{\Vert \xi_\ell - \xi^\ast\Vert}=0.
$ Thus, by Proposition~\ref{prop:13}, the Armijo step size rule eventually selects $\tau_\ell=1$ for all $\ell$ sufficiently large.
\end{proof}
\subsubsection{The proof of Theorem~\ref{thm:converge_alg2} (iii)}
We begin with the result of the local convergence of the linear Newton iteration.
\begin{proposition}[{\cite[Thm.~7.5.15]{Facchinei+Pang-Finivariineqcomp:03}}]\label{thm:converge}
Let $\Phi \colon \R^n\rightarrow \R^n$ be locally Lipschitz continuous and admit a (strong) LNA scheme
$S$ at $\eta^\ast\in \R^n$ such that $\Phi(\eta^\ast)=0$. If every matrix
$H\in S(\eta^\ast)$ is nonsingular, then the linear Newton method%~\eqref{gn}
\begin{equation*}
     \eta^{k+1}= \eta^{k} - H^{-1}_k\Phi(\eta^k) , \mbox{ with }H_k \in S(\eta^k),\quad k=0,1,\ldots,
\end{equation*}
 converges Q-superlinearly (resp. Q-quadratically)
to the solution $\eta^\ast$ provided that $\eta^0$ is sufficiently close to $\eta^\ast$.
\end{proposition}

\begin{proof}[The proof of Theorem~\ref{thm:converge_alg2} (iii)]
Theorem~\ref{thm:converge_alg2} (ii) implies that Algorithm~\ref{alg:2} eventually accepts the unit step length in the line search procedure, i.e., there exists $\ell^\prime$ such that
\begin{equation*}
  \xi_{\ell+1} = \xi_\ell - V_\ell^{-1}\nabla \psi(\xi_\ell),\quad V_\ell\in T(\xi_\ell),\quad  \forall \ell \ge \ell^\prime.
\end{equation*}
By Proposition~\ref{prop:LNA}, we know that $T(\xi)$ is a (strong) LNA of $\nabla \psi$ at $\xi^\ast$. Thus,
Theorem~\ref{thm:converge_alg2} (iii) follows from Proposition~\ref{thm:converge}.
\end{proof}

\subsection{Implementation of Algorithm~\ref{alg:2}}
At the end of this section, we illustrate the implementation of Algorithm~\ref{alg:2} using
the $\ell^1$-TV regularization as an example. We only describe the construction of the gradient $\nabla \psi(\xi)$ and the linear Newton approximation scheme $T(\xi)$.
\subsubsection{$\ell^1$\textendash tv regularization}
Let $F$ be an image contaminated by salt-and-pepper noise. Let $y \in \R^{n^2}$ be a column vector collecting the columns of $F$, i.e.,
\begin{equation*}
  y_{i+n(j-1)} = F_{i,j},\quad  1\le i \le n ,  1\le j\le n.
\end{equation*}
Our object is to remove noise from the image $F$. It can be accomplished by minimizing the function defined by
\begin{equation*}
  \min_{u}\alpha \Vert u - y\Vert_1 +  \Vert \nabla u\Vert_{2,1},
\end{equation*}
where $\Vert z\Vert_{2,1}$ for $z\in\R^{2n^2}$ denotes the norm $\Vert z\Vert_{2,1} = \sum_{i=1}^{n^2}\sqrt{z_i^2 +z_{i+n^2}^2}$,
$\alpha$ is the regularization parameter, and $\nabla$ denotes the (discretized) gradient matrix $ \nabla = [D_1^{\top},D_2^{\top}]^{\top} \in \R^{2n^2\times n^2}$. Here, $D_1 = I_n\otimes D$ and $D_2 = D\otimes I_n$ with $D$ being the one-dimensional finite difference matrix with periodic boundary condition and $I_n$ being the identity matrix with $n\times n$.

The problem is written in the form of \eqref{min} by setting
\begin{equation*}
  f(u)=0,\quad \phi(v,z) = \alpha\Vert v-y\Vert_1+\Vert z \Vert_{2,1},
  \quad \mbox{and}\quad
E=\begin{bmatrix}
I_{n^2} \\ \nabla
\end{bmatrix}.
\end{equation*}
The proximal mappings of the function $\phi$ are composed of the proximal mappings of  $\ell^1$ norm transported by $y$ and $\Vert \cdot\Vert_{2,1}$ norm;
\begin{equation*}
  \prox{\phic}(v,z) = \begin{bmatrix}
    \prox{\frac{\alpha}{c}\Vert\cdot-y\Vert_1}(v)\\
    \prox{\Vert \cdot\Vert_{2,1}/c}(z)
  \end{bmatrix}.
\end{equation*}
Now let $u_k\in \R^{n^2}$, $\lambda_k$ be the $k$-th iterates of the
outer iteration, and $\xi_\ell$ be the current solution of the inner iteration. Here, $\lambda_k=(\lambda^{(1)}_k,\lambda^{(2)}_k)$, $\lambda^{(1)}_k\in \R^{n^2}$ is the Lagrange multiplier for $\prox{\frac{\alpha}{c}\Vert\cdot-y\Vert_1}(v)$, and $\lambda^{(2)}_k \in \R^{2n^2}$ is that for $\prox{\Vert \cdot\Vert_{2,1}/c}(z)$. The building blocks for determining the direction $d_\ell$ are:
\begin{align*}
  \nabla \psi(\xi_\ell) &= c_kE^{\top}(E \xi_{\ell} +{\lambda_k}/{c_k} - \prox{\phi/{c_k}}(E \xi_{\ell} + {\lambda_k}/{c_k})) + c^{-1}_k(\xi_{\ell} - u_k), \\
   V_\ell &=  c_k E^{\top} (I_{3n^2}-G)E + c^{-1}_k I_{n^2}.
\end{align*}
where $G \in \partial_B( \prox{\phic})(v,z)$ with $v =\xi_\ell +\lambda^{(1)}_k/c_k$ and $z =\nabla \xi_\ell +\lambda^{(2)}_k/c_k$.
Below, we provide closed forms of $\prox{\phic}(v,z)$ and $G$.

The proximal mapping $\prox{\frac{\alpha}{c}\Vert\cdot\Vert_1}(z)$ for $z\in\R^n$ is the well-known soft-thresholding operator
\begin{equation*}
  \prox{\frac{\alpha}{c}\Vert\cdot\Vert_1}(z) = [\prox{\frac{ \alpha}{c}\vert \cdot \vert}(z_1), \ldots,\prox{\frac{ \alpha}{c}\vert \cdot \vert}(z_n)]^{\top}.
\end{equation*}
Here, $\prox{\frac{ \alpha}{c}\vert \cdot \vert}(s) = \max(s-\tfrac{ \alpha}{c} ,\min(s+\tfrac{ \alpha}{c} ,0))$ for $s\in\R$.
By \eqref{composition}, we have
$\prox{\frac{\alpha}{c}\Vert\cdot-y\Vert_1}(z) = \prox{\frac{\alpha}{c}\Vert\cdot\Vert_1}(z-y)+y \in \R^n$.
A limiting Jacobian $G_1\in \partial_B(\prox{\frac{\alpha}{c}\Vert\cdot-y\Vert_1})(z)$ is diagonal matrix given by
\begin{equation*}
[G_1]_{j,j} =\left\{
\begin{array}{ll}
1 &  \mbox{ if } \vert z_j-y_j\vert  > \frac{ \alpha}{c} ,\\
\{0,1\} & \mbox{ if } \vert z_j -y_j\vert= \frac{ \alpha}{c}, \\
0 & \mbox{ otherwise. }
\end{array} \right.
\end{equation*}
On the other hand, the proximal mapping $\prox{\frac{\Vert \cdot \Vert_{2,1}}{c}}(z)$  is block-separable, and its $i$-th and $(i+n^2)$-th components ($i=1,\ldots,n^2$) are given by
\begin{equation*}
([\prox{\frac{\Vert \cdot \Vert_{2,1}}{c}}(z)]_{i},[\prox{\frac{\Vert \cdot \Vert_{2,1}}{c}}(z)]_{i+n^2}) = \left\{
\begin{array}{ll}
(z_{i}-\frac{z_{i}}{c r_i},z_{i+n^2}-\frac{z_{i+n^2}}{c r_i})& \mbox{ if }  c r_i \ge 1, \\[5pt]
0 & \mbox{ if } c r_i < 1, \\
\end{array}\right.
\end{equation*}
where $r_i=\sqrt{z_{i}^2 + z_{i+n^2}^2}$.  A limiting Jacobian $G_2\in\partial_B
(\prox{\frac{\Vert \cdot \Vert_{2,1}}{c}})(z)$ for $z\in \R^{2n^2}$ is composed of four diagonal matrices
\begin{align*}
G_2=\begin{bmatrix}
  G_{11} & G_{12} \\ G_{21} & G_{22}
\end{bmatrix}
\quad \mbox{with}\quad
\begin{array}{ll}
G_{11} = \mbox{diag}(\{g_{i,i}\}_{i=1}^{n^2}),   &  G_{12}=\mbox{diag}(\{g_{i,n^2+i}\}_{i=1}^{n^2})  \\
G_{21} = \mbox{diag}(\{g_{n^2+i,i}\}_{i=1}^{n^2}) ,&G_{22}=\mbox{diag}(\{g_{n^2+i,n^2+i}\}_{i=1}^{n^2})
\end{array}
\end{align*}
where the $(i,i), (i,i+n^2),(i+n^2,i),(i+n^2,i+n^2)$ components of $G\in \R^{2n^2\times2n^2}$ are given by
\begin{align*}
\begin{bmatrix}
g_{i,i} & g_{i,i+n^2}\\
g_{i+n^2,i} & g_{i+n^2,i+n^2}
\end{bmatrix}
= \left\{\begin{array}{cc}
\displaystyle I_2- \frac{1}{c r_i^{3}}
\begin{bmatrix}
z_{i+n^2}^2& -z_iz_{i+n^2} \\
-z_iz_{i+n^2} & z_i^2
\end{bmatrix} & \mbox{ if } c r_i > 1 \\[10pt]
0  & \mbox{ if } c r_i< 1,
\end{array}\right.
\end{align*}
and any of these elements if $c r_i = 1$ (see, e.g., \cite[Sect.~5.2.3]{Patrinos+StellaETAL-ForwtrunNewtmeth:14}).
Consequently, we have
\begin{align*}
  \nabla \psi(\xi_\ell)
  &=c_k\left[z^{(1)}_\ell -\prox{\frac{\alpha}{c_k}\Vert \cdot \Vert_{1}}(z_\ell^{(1)})
    + \nabla^{\top}(z^{(2)}_\ell- \prox{\frac{\Vert \cdot \Vert_{2,1}}{{c_k}}}(z^{(2)}_\ell)\right]+ c^{-1}_k(\xi_{\ell} - u_k), \\
   V_\ell
   & = c_k( (I_{n^2}-G_1) + \nabla^{\top} (I_{2n^2}-G_2)\nabla) + c^{-1}_k I_{n^2},
\end{align*}
where $z_\ell^{(1)} = \xi_\ell+\lambda^{(1)}_k/c_k$ and $z_\ell^{(2)} = \nabla \xi_\ell+\lambda^{(2)}_k/c_k$.
\section*{Conclusion}
In this paper, we have developed the proximal method of multipliers for a class of nonsmooth convex optimization problems
arising in various application domains and a Newton-type method for solving the subproblems.
We provided a rigorous proof on the global convergence of the Newton method with line search and on the rate of the convergence.

To make the proposed framework applicable to real-word applications,
further studies are needed on several important issues, including the
development of efficient solvers for the (possibly) large linear
system (Newton system) and providing parameter choice rules for $c_k$
and $\epsilon_k$, which may have a great influence on the numerical performance of the algorithm. These issues will be investigated in future work.
%\section*{Acknowledgement}
%This research is supported by the New Energy and Industrial Technology Development Organization (NEDO).
%\section*{Appendix}

%\subsection{Evaluation of the subgradient}

\renewcommand{\theequation}{A.\arabic{equation}}
\setcounter{equation}{0}
\section*{Appendix}

\renewcommand{\theequation}{A.\arabic{equation}}
\setcounter{equation}{0}
\subsection*{A. Augmented Lagrangian method (method of multipliers)}
This section provides a brief review of the augmented Lagrangian method. %the initial work of Hestenes and Powell.
%The description presented below is not meant to be precise or rigorous;
%our aim here is to dredge up the augmented Lagrangian method proposed
%by Fortin \cite{Fortin-Minisomenon-func:75} for the problem
%\eqref{min}.% which has been buried and has been scarcely referred to in the subsequence literature since its appearance.

\noindent
\textbf{Augmented Lagrangian by Hestenes and Powell}\\
Initially, the method of multipliers was independently proposed by Hestenes \cite{Hestenes-Multgradmeth:69} and Powell \cite{Powell-methnonlconsmini:69} for solving nonlinear programming problems with equality constraints
\begin{equation*}
  \min_{x\in \R^n} f(x) \quad \mbox{subject to}\quad h(x) = 0,
\end{equation*}
where $f$ and $h$ are smooth functions. The method of multipliers performs a sequence of minimization problems
\begin{equation*}
  x_{k} =\arg \min_{x} \left\{ L_{c_k}(x,\lambda_k) = f(x) + (\lambda_k, h(x)) + \frac{c_k}{2}\Vert h(x) \Vert^2 \right\},
\end{equation*}
followed by the multiplier update
\begin{equation*}
  \lambda_{k+1} = \lambda_k + c_k \nabla_\lambda L_{c_k}(x_k,\lambda_k)=\lambda_k + c_k h(x_k).
\end{equation*}
The sequence $\{c_k\}$ may be either fixed a priori or adaptively increased during the iteration.

\noindent
\textbf{Augmented Lagrangian by Rockafellar}\\
This approach was generalized to nonlinear programming problems with equality and inequality constraints over a convex set $C\subset \R^n$
\begin{equation}\label{ineq}
  \min_{x\in C} f(x) \quad \mbox{subject to}\quad h(x)=0,\quad g(x) \le 0,
\end{equation}
by Rockafellar \cite{Rockafellar-dualapprsolvnonl:73,Rockafellar-multmethHestPowe:73}.  Hereafter, we
focus on the inequality constraint only to keep the presentation simpler.
The augmented Lagrangian, which is called the \textit{penalty Lagrangian} in \cite{Rockafellar-dualapprsolvnonl:73},  for the problem is derived by casting the inequality problem into the equality-constrained problem using a slack variable $v$
\begin{equation*}
  \min_{x\in C,v\in \R^m} f(x)+ \phi(v)  \quad \mbox{subject to}\quad g(x) - v = 0
\end{equation*}
where $\phi(v)$ is the indicator function of the set $K=\{v \in \R^m \mid v_i \le 0, \forall i \}$. Then, define an augmented  Lagrangian $\mathcal{L}_{c}(x,v,\lambda)$ by
\begin{align*}
\mathcal{L}_{c}(x,v,\lambda)&= f(x) + \phi(v)+ (\lambda, g(x)-v) + \frac{c}{2}\Vert g(x)-v\Vert^2.
\end{align*}
The augmented Lagrangian $L_c(x,\lambda)$ for the problem \eqref{ineq} is defined by eliminating the slack variable from $\mathcal{L}_c(x,v,\lambda)$ by marginalization with respect to $v$
\begin{align*}
  \AL{x}{\lambda}{c} &:=\min_{v\in \R^m}\mathcal{L}_{c}(x,v,\lambda)\\
  &=f(x) + \min_{v} \left(\phi(v)+ \frac{c}{2}\Vert g(x)+\lambda/c-v\Vert^2\right) - \frac{\Vert \lambda\Vert^2}{2c}\\
& = f(x) + \phi_c(g(x)+\lambda/c) - \frac{\Vert \lambda\Vert^2}{2c}.
\end{align*}
The proximal operator for the indicator function $\phi$ is the projection onto the convex set $K$,
and it is explicitly given as $ \prox{\phic} = \min(z,0)$,
and thus, the Moreau envelope $\phi_c(v)$ is written as
\begin{align*}
 \phi_c(v)= \phi(v(g(x)+\lambda/c)) + \frac{c}{2}\Vert g(x)+\lambda/c - \min(g(x)+\lambda/c,0)\Vert^2=\frac{c}{2}\Vert \max( g(x)+\lambda/c,0)\Vert^2.
\end{align*}
Thus, we obtain
\begin{align*}%\label{ag_ineq}
  \AL{x}{\lambda}{c} =f(x)+ \frac{c}{2}\Vert \max( g(x)+\lambda/c,0)\Vert^2-\frac{\Vert\lambda\Vert^2}{2c}.
\end{align*}
The method of multipliers is then written as
\begin{align}
  x_{k+1} &= \arg\min_{x\in C} \AL{x}{\lambda_k}{c_k}, \label{step1_alm_ineq} \\
  \lambda_{k+1}&= \lambda_k + c_k\nabla_\lambda \AL{x_{k+1}}{\lambda_k}{c_k} \nonumber \\
  &= \lambda_k + c_k(g(x_{k+1})-\min(g(x_{k+1})+\lambda_k/c_k,0)) \nonumber \\
  & = \max(\lambda_k + c_k g(x_{k+1}),0). \nonumber
\end{align}
In practical situations, the minimization \eqref{step1_alm_ineq} can be performed only inexactly.
Rockafellar's approach \cite{Rockafellar-MonooperaugmLagr:78} allows
one to evaluate $x_{k+1}$ approximately provided that a certain condition on the accuracy of the approximation is fulfilled. In addition, he showed that increasing the parameter $c_k$ in every step improves the rate of convergence of the algorithm.
The method of multipliers for the equality and inequality constraints is studied in extensive detail by Bertsekas \cite{MR690767}.

\noindent
\textbf{Augmented Lagrangian by Glowinski-Marroco and Gabay-Mericer}\\
Glowinski and Marroco \cite{glowinski1975} carried out a similar approach to the problem~\eqref{min} by transforming the problem into the equality-constrained convex optimization problem
\begin{equation*}
  \min_{x\in {\R^n},v\in {\R^m}} f(x) + \phi(v) \quad \mbox{ subject to } \quad E x = v.
\end{equation*}
An augmented Lagrangian $\mathcal{L}_c(x,v,\lambda)$ is defined by
\begin{equation}\label{ag_glowinski}
\mathcal{L}_c(x,v,\lambda) = f(x) + \phi(v) + (\lambda,E x - v) + \frac{c}{2}\Vert E x - v\Vert^2.
\end{equation}
The method of multipliers is written as
\begin{align*}
  (x_{k+1},v_{k+1}) &= \arg \min_{x,v} \mathcal{L}_c(x,v,\lambda_k), \\%\label{step1_alm_glowinski} \\
  \lambda_{k+1} & = \lambda_k + c\nabla_\lambda \mathcal{L}_c(x_{k+1},v_{k+1},\lambda_k)=\lambda_k + c(Ex_{k+1} - v_{k+1}).
\end{align*}
Gabay and Mericer \cite{Gabay+Mercier-dualalgosolunonl:76a} proposed decoupling the simultaneous minimization over $x$ and $v$ into the minimization problems with respect to $x$ followed by $v$;
\begin{align*}
  x_{k+1} &= \arg \min_{x} \mathcal{L}_c(x,v_k,\lambda_k), \\%\label{step1_admm} \\
v_{k+1} &= \arg \min_{v} \mathcal{L}_c(x_{k+1},v,\lambda_k) = \prox{\phic}(Ex_{k+1}+\lambda_k/c), \\%\label{step2_admm} \\
  \lambda_{k+1} & = \lambda_k + c\nabla_\lambda \mathcal{L}_c(x_{k+1},v_{k+1},\lambda_k)=\lambda_k + c(Ex_{k+1} - v_{k+1}).
\end{align*}
The algorithm is called \textit{alternating direction method of multipliers (ADMM)}, which is frequently used in a wide range of scientific communities, such as numerical partial differential equations \cite{Fortin+Glowinski-AugmLagrmeth:83} and statistical learning \cite{Boyd+ParikhETAL-DistOptiStatLear:11}, to name a few.

\noindent
\textbf{Augmented Lagrangian by Fortin}\\
Fortin \cite{Fortin-Minisomenon-func:75} proposed an alternative Lagrangian obtained by eliminating $v$ from
\eqref{ag_glowinski}.
\begin{align}\label{ag_fortin}
\AL{x}{\lambda}{c}&=\min_{v}\mathcal{L}_c(x,v,\lambda)=f(x) + \min_{v} \left(\phi(v)+ \frac{c}{2}\Vert Ex+\lambda/c-v\Vert^2\right) - \frac{\Vert \lambda\Vert^2}{2c}.
\end{align}
Using the Lagrangian $L_c$, he developed an algorithm for solving the problem \eqref{min}.
By following exactly the same argument for the inequality-constrained problem,
one obtains the expression of $L_c$ (see \eqref{ag})
and the method of multipliers, which is given as
\begin{align}
  x_{k+1} &= \arg\min_{x\in {\R^n}} \AL{x}{\lambda_k}{c}, \label{step1_alm_fortin} \\
  \lambda_{k+1}&= \lambda_k + c\nabla_\lambda \AL{x_{k+1}}{\lambda_k}{c}= \lambda_k + c(Ex_{k+1}-\prox{\phic}(Ex_{k+1} +\lambda_k/c)) \nonumber\\% \label{step2-1_alm_fortin}\\
  & = \prox{c\phi^\ast}(\lambda_k + c Ex_{k+1}). \label{step2-2_alm_fortin}
\end{align}
The equality in \eqref{step2-2_alm_fortin} follows from the Moreau decomposition, $\prox{\phic}(z)+c^{-1}\prox{c\phi^\ast}(cz)=z$.

 For a particular case where $\phi$ is the indicator function of the convex set $K=\R^n_{-}=\{x\in \R^n\mid x\le 0\}$,
 the proximal mapping is the projection onto the set $K$, i.e., $\prox{\phic}(z) = \min(z,0)$. Therefore,
 the augmented Lagrangian \eqref{ag} is identical to \eqref{ag_ineq} with $g(x)=Ex$.
 This implies that the augmented Lagrangian \eqref{ag_fortin}
 is the legitimate generalization of Rockafellar's augmented Lagrangian.

 The method was initially developed for the unconstrained problem \eqref{min}, and it was extended to the problem with an abstract constraint $x\in C$ in \cite{Ito+Kunisch-AugmLagrmethnons:00}, where the convergence of the algorithm (with the presence of the constraint $x\in C$) is established under the strict assumption that the exact minimization for $x_{k+1}$ is available.
\renewcommand{\theequation}{B.\arabic{equation}}
\setcounter{equation}{0}
\subsection*{B. Dual-Augmented Lagrangian (DAL)}
The dual-augmented Lagrangian (DAL) algorithm proposed in \cite{Tomioka+Sugiyama-DualLagrMethEffi:09} and analyzed in-depth in \cite{Tomioka+SuzukiETAL-SupeConvDualAugm:11} is an efficient numerical solver for the sparse optimization problem of estimating an $n$-dimensional parameter vector $x$ from $m$ training examples
\begin{equation}\label{min_dal}
f_\ell(Ax) + \phi(x).
\end{equation}
Here, $f_\ell$ is a $1/\nu$ smooth loss function, e.g., the square
loss, the logistic loss, the hyperbolic secant likelihood, and
multiclass logit (see Table 1 in
\cite{Tomioka+SuzukiETAL-SupeConvDualAugm:11}), and $A\in \R^{m\times
  n}$ is a design matrix. The prior $\phi$ is a proper, convex
function introducing sparseness into the solution; such function
includes the $\ell_1$ norm for sparse estimation, the group norm for grouped lasso and the trace norm for a low rank matrix estimation (see Table 2 in \cite{Tomioka+SuzukiETAL-SupeConvDualAugm:11}).
Each of the convex conjugate functions of these priors listed above is an indicator function of a convex set, and hence, the proximal operator of $\phi^\ast$ is the projection onto the convex set; hence, it holds that
  \begin{equation}\label{ass_dal}
   \phi^\ast( \prox{c\phi^\ast}(z) ) = 0,\quad \forall z\in {\R^m}, \forall c>0.
  \end{equation}

The DAL algorithm is the method of multipliers applied to the Fenchel dual problem
 \begin{equation}\label{min_dual}
   \min_\alpha f^{\ast}_{\ell}(-\alpha) + \phi^\ast(A^\top \alpha),
 \end{equation}
and the algorithm is described as follows:
\begin{align}
\alpha_k & \approx \arg \min_{\alpha}  f^{\ast}_{\ell}(-\alpha) + 1/(2c_k)\Vert \mbox{prox}_{c_k\phi}(w_k+c_kA^\top \alpha)\Vert^2, \label{step1_dal} \\
w_{k+1} & = \mbox{prox}_{c_k\phi}(w_k+c_kA^\top \alpha_k), \label{step2_dal}
\end{align}
where the approximation $\alpha_k$ must satisfy the following criterion
\begin{equation*}
\Vert \nabla_\alpha \AL{\alpha_k}{w_k}{c_k} \Vert \le \sqrt{\nu/c_k}\Vert w_{k+1}-w_k\Vert.
\end{equation*}
Here, $c_k$ is an increasing positive parameter prescribed prior to the execution of the algorithm, e.g., $c_k = 2^k$.
It is shown in \cite{Tomioka+SuzukiETAL-SupeConvDualAugm:11} that the sequence $\{w_k\}$ generated by DAL converges to a solution of \eqref{min_dal}.

The derivation of \eqref{step1_dal}, \eqref{step2_dal} presented in \cite{Tomioka+Sugiyama-DualLagrMethEffi:09} is similar to the argument deriving $L_c$ in \eqref{ag_fortin} from $\mathcal{L}_c$ in \eqref{ag_glowinski}.
Here, for the reader's convenience, we provide a derivation of DAL.
Let $\AL{\alpha}{w}{c}$ be the augmented Lagrangian of the dual problem \eqref{min_dual}
\begin{equation}\label{ag_dual}
\AL{\alpha}{w}{c}= f^\ast_{\ell}(-\alpha) + \phi_c^\ast(A^\top \alpha + w/c) - \frac{\Vert w\Vert^2}{2c}.
\end{equation}
From the Moreau decomposition $\prox{c\phi}(z)+c\prox{\phi^\ast/c}(z/c)=z$ and the assumption \eqref{ass_dal}, we obtain
\begin{align*}
\phi^\ast_c(z)  &=\phi^\ast(\prox{\phi^\ast/c}(z)) + \frac{c}{2}\Vert \prox{\phi^\ast/c}(z)-z\Vert^2 \\
&=\frac{c}{2}\Vert \prox{\phi^\ast/c}(z)-z\Vert^2=\frac{c}{2}\Vert \prox{\phi^\ast/c}(cz/c)-cz/c\Vert^2=\frac{1}{2c}\Vert \prox{c\phi}(cz)\Vert^2,
\end{align*}
for all $z\in {\R^m}$. Thus, we see that the minimization step \eqref{step1_alm_fortin} is written as
\begin{equation*}
\arg\min_{\alpha} L_c(\alpha,w_k)=\arg\min_{\alpha} f^\ast_{\ell}(-\alpha) + \frac{1}{2c_k}\Vert \prox{c\phi}(c_k(A^\top \alpha + w_k/c_k))\Vert^2,
\end{equation*}
which indicates that DAL is the augmented Lagrangian method by \eqref{ag_dual}. From this perspective, DAL can be interpreted as a special instance of Fortin's augmented Lagrangian method.

\renewcommand{\theequation}{C.\arabic{equation}}
\setcounter{equation}{0}
\subsection*{C. Forward-Backward Newton Method}
Let us consider the minimization problem of a composite function, which is a special instance of \eqref{min} with $E=I$.
\begin{equation}\label{min2}
\min_{x\in \R^n} F(x):=f(x)   +\phi(x).
\end{equation}
The optimality system for this problem is written as (cf., e.g.,  \cite[Ch.~4]{Ito+Kunisch-AugmLagrmethnons:00}, \cite{Jin+Takeuchi-Lagroptisystclas:16})
\begin{equation*}
  \nabla f(x) + \lambda=0\quad \mbox{and}\quad x-\prox{\phic}(x+\lambda/c)=0.
\end{equation*}
Eliminating $\lambda$ from the optimality system yields the well-known optimality condition for \eqref{min2}:
\begin{equation}\label{fbn_opt}
 x-\prox{\phic}(x-\nabla f(x)/c)=0.
\end{equation}
The forward-backward Newton method proposed in \cite{parikh2013proximal} and further investigated in \cite{Patrinos+StellaETAL-ForwtrunNewtmeth:14} is a nonsmooth Newton method applied to the equation \eqref{fbn_opt}.
\begin{align*}%\label{pn}
\left(I - G_k(I -c^{-1}\nabla^2f(x_k))\right)d_k&= -(x_k - \prox{\phic}(z_k)), \\
x_{k+1} &= x_k + \tau_k d_k.%\nonumber
\end{align*}
Here, $x_k$ is the current iterate, $\tau_k>0$ is a step size to be selected appropriately, $z_k= x_k -\nabla f(x_k)/c$, and $G_k$ is a generalized Jacobian of ${\prox{\phic}}(z)$ at $z=z_k$.
If $c$ is larger than the Lipschitz constant $L_f$ of the gradient $\nabla f$, then
the Newton system is nonsingular and the direction $d_k$ is uniquely determined, which is shown to be a descent direction of the function defined using the augmented Lagrangian
\begin{equation*}%\label{fbe}
F_c(x):=\AL{x}{-\nabla f(x)/c}{c} = f(x) + \phi_c(x -\nabla f(x)/c)-1/(2c)\Vert \nabla f(x)\Vert^2.
\end{equation*}
The function $F_c$, which is called \textit{forward-backward envelope (FBE)} in \cite{Patrinos+StellaETAL-ForwtrunNewtmeth:14}, is used as a merit function for determining the
step length $\tau_k$ in the Armijo line search algorithm.

To avoid the possible ill-conditioning of the Hessian $\nabla^2 f(x)$, a regularization term $\delta_k I$, where $\delta_k$ is an appropriately chosen positive constant, is performed, i.e., the direction $d_k$ is determined by
 \begin{align}\label{pn2}
\left(I - G_k(I -c^{-1}\nabla^2f(x_k)) + \delta_k I\right)d&= -(x_k - \prox{\phic}(z_k))
\end{align}
The article  \cite{Patrinos+StellaETAL-ForwtrunNewtmeth:14} presents the proof of the convergence of the algorithm (\eqref{pn2} followed by the line search), the extensive numerical tests and a list of the generalized Jacobian of the proximal operator of directly relevant to the sparse optimization.
\begin{remark}
The function $F_c$ seems to have appeared for the first time in the monograph \cite[p.~709, p.~787]{Facchinei+Pang-Finivariineqcomp:03} for the minimization problem of a smooth convex function $f$ over a convex set $K$ (hence, the proximal operator is the projection on a convex set $K$), where the direction $d_k$ is determined by $d_k = x_k -\Pi_K(x_k-\nabla f(x_k))$ and a step size $\tau_k$ is selected by the Armijo line search applied to the function $F_c$.
The work \cite{Patrinos+StellaETAL-ForwtrunNewtmeth:14} can be viewed as the generalization: the projection $\Pi_K$ is replaced by  the proximal mapping $\prox{\phi/c}$, and the direction $d_k$ is determined by the Newton-like method \eqref{pn2}.
\end{remark}

%\bibliographystyle{siam}
%\bibliography{fortin}

\end{document}